\documentclass{birkjour}
\usepackage{amsfonts}
\usepackage{amsthm}
\usepackage{amssymb}
\usepackage{amsmath}
\usepackage{upref}
\usepackage{mathrsfs}

\newtheorem{thm}{Theorem}[section]
\newtheorem{lemma}[thm]{Lemma}
\newtheorem{prop}[thm]{Proposition}
\newtheorem{cor}[thm]{Corollary}
\theoremstyle{definition}
\newtheorem{defin}[thm]{Definition}
\theoremstyle{remark}
\newtheorem{rem}[thm]{Remark}

\numberwithin{equation}{section}

\newcommand{\dif}{\mathrm{d}}
\newcommand{\totdif}{\mathrm{D}}

\newcommand{\mc}{C}
\newcommand{\mf}{\mathscr{F}}
\newcommand{\ml}{L}

\newcommand{\mr}{\mathbb{R}}

\newcommand{\prst}{\mathbb{P}}
\newcommand{\stred}{\mathbb{E}}
\newcommand{\ind}{\mathbf{1}}
\newcommand{\mn}{\mathbb{N}}

\newcommand{\mt}{\mathbb{T}}

\DeclareMathOperator{\supt}{supp}
\DeclareMathOperator*{\esssup}{ess\,sup}

\DeclareMathOperator{\diver}{div}
\newcommand{\tec}{{\overset{\cdot}{}}}

\newcommand{\bl}{\big\langle}
\newcommand{\bR}{\big\rangle}

\renewcommand{\labelenumi}{(\roman{enumi})}
\begin{document}

%\begin{frontmatter}

\title[Degenerate Parabolic SPDE\MakeLowercase{s}]{Degenerate Parabolic Stochastic Partial\\Differential Equations}
\author{Martina Hofmanov\'a}%\corref{cor}\fnref{fn}}
%\ead{martina.hofmanova@bretagne.ens-cachan.fr}

\address{Department of Mathematical Analysis\\ Faculty of Mathematics and Physics, Char\-les University\\ Sokolovsk\'a~83\\ 186~75 Praha~8\\ Czech Republic\vspace{2mm}}
\address{Institute of Information Theory and Automation of the ASCR\\ Pod~Vod\'arenskou v\v{e}\v{z}\'i~4\\ 182~08 Praha~8\\ Czech Republic\vspace{2mm}}
\address{IRMAR, ENS Cachan Bretagne, CNRS, UEB\\ av. Robert Schuman\\ 35~170 Bruz\\ France}
%\cortext[cor]{Correspondence address: Institute of Information Theory and Automation of the ASCR, Pod~Vod\'arenskou v\v{e}\v{z}\'i~4,\ 182~08 Praha~8, Czech Republic}
%\fntext[fn]{This research was supported in part by the GA\,\v{C}R Grant no. P201/10/0752.}

\email{martina.hofmanova@bretagne.ens-cachan.fr}
\thanks{This research was supported in part by the GA\,\v{C}R Grant no. P201/10/0752.}

\begin{abstract}
We study the Cauchy problem for a scalar semilinear dege\-nerate parabolic partial differential equation with stochastic forcing. In particular, we are concerned with the well-posedness in any space dimension. We adapt the notion of kinetic solution which is well suited for degenerate parabolic problems and supplies a good technical framework to prove the comparison principle. The proof of exis\-tence is based on the vanishing viscosity method: the solution is obtained by a compactness argument as the limit of solutions of nondegenerate approximations.
\end{abstract}

%\begin{keyword}
%degenerate parabolic stochastic partial differential equation \sep kinetic solution
%%\MSC[2010] fdjkahf
%\end{keyword}

%\end{frontmatter}
\maketitle

\section{Introduction}

In this paper, we study the Cauchy problem for a scalar semilinear degenerate parabolic partial differential equation with stochastic forcing
\begin{equation}\label{rovnice}
\begin{split}
\dif u+ \diver \big(B(u)\big)\dif t&=\diver\big(A(x)\nabla u\big)\dif t + \varPhi(u)\,\dif W,\quad x\in\mt^N,\;t\in(0,T),\\
u(0)&=u_0,
\end{split}
\end{equation}
where $W$ is a cylindrical Wiener process. Equations of this type are widely used in fluid mechanics since they model the phenomenon of convection-diffusion of ideal fluid in porous media. Namely, the important applications including for instance two or three-phase flows can be found in petroleum engineering or in hydrogeology. For a thorough exposition of this area given from a practical point of view we refer the reader to \cite{petrol} and to the references cited therein.

The aim of the present paper is to establish the well-posedness theory for solutions of the Cauchy problem \eqref{rovnice} in any space dimension. Towards this end, we adapt the notion of kinetic formulation and kinetic solution which has already been studied in the case of hyperbolic scalar conservation laws in both deterministic (see e.g. \cite{vov}, \cite{lpt1}, \cite{lions}, \cite{pert1}, or \cite{perth} for a general presentation) and stochastic setting (see \cite{debus}); and also in the case of deterministic degenerate parabolic equations of second-order (see \cite{chen}). To the best of our knowledge, in the degenerate case, stochastic equations of type \eqref{rovnice} have not been studied yet, neither by means of kinetic formulation nor by any other approach.

The concept of kinetic solution was first introduced by Lions, Perthame, Tadmor in \cite{lions} for deterministic first-order scalar conservation laws and applies to more general situations than the one of entropy solution as considered for example in \cite{car}, \cite{feng}, \cite{kruzk}. Moreover, it appears to be better suited particularly for degenerate parabolic problems since it allows us to keep the precise structure of the parabolic dissipative measure, whereas in the case of entropy solution part of this information is lost and has to be recovered at some stage. This technique also supplies a good technical framework to prove the $L^1$-comparison principle which allows to prove uniqueness. Nevertheless, kinetic formulation can be derived only for smooth solutions hence the classical result \cite{gyongy} giving $L^p$-valued solutions for the nondegenerate case has to be improved (see \cite{hof}).

In the case of hyperbolic scalar conservation laws, Debussche and Vovelle \cite{debus} defined a notion of generalized kinetic solution and obtained a comparison result showing that any generalized kinetic solution is actually a kinetic solution. Accordingly, the proof of existence simplified since only weak convergence of approximate viscous solutions was necessary.

The situation is quite different in the case of parabolic scalar conservation laws. Indeed, due to the parabolic term, the approach of \cite{debus} is not applicable: the comparison principle can be proved only for kinetic solutions (not generalized ones) and therefore strong convergence of approximate solutions is needed in order to prove the existence. Moreover, the proof of the comparison principle itself is much more delicate then in the hyperbolic case. %then for the hyperbolic conservation laws.

%The letter $C$ will denote a positive constant, which is unimportant and may change from one line to another.

The exposition is organised as follows. In Section 2 we review the basic setting and define the notion of kinetic solution. Section 3 is devoted to the proof of uniqueness. We first establish a technical Proposition \ref{propdubling} which then turns out to be the keystone in the proof of comparison principle in Theorem \ref{uniqueness}. We next turn to the proof of existence in Sections 4 and 5. First of all in Section 4, we make an additional hypotheses upon the initial condition and employ the vanishing viscosity method. In particular, we study certain nondegenerate problems and establish suitable uniform estimates for the corresponding sequence of approximate solutions. The compactness argument then yields the existence of a martingale kinetic solution which together with the pathwise uniqueness gives the desired kinetic solution (defined on the original stochastic basis). In the final section, the existence of a kinetic solution is shown for general initial data.

We note that an important step in the proof of existence, identification of the limit of an approximating sequence of solutions, is based on a new general method of constructing martingale solutions of SPDEs (see Propositions \ref{ident1}, \ref{ident2} and the sequel), that does not rely on any kind of martingale representation theorem and therefore holds independent interest especially in situations where these representation theorems are no longer available. First applications were already done in \cite{on1}, \cite{on2} and, in the finite-dimensional case, also in \cite{hofse}.

% In particular, the kinetic formulation is based on consideration of a new real-valued variable $\xi$, called velocity, so that the unknown becomes a density-like kinetic function $f$.
% the correction term which comes from the application of the It\^o formula (see \cite{prato} for a deeper discussion)

\section{Notation and main result}

\label{notation}

We now give the precise assumptions on each of the terms appearing in the above equation \eqref{rovnice}. We work on a finite-time interval $[0,T],\,T>0,$ and consider periodic boundary conditions: $x\in\mt^N$ where $\mt^N$ is the $N$-dimensional torus.
The flux function
$$B=(B_1,\dots,B_N):\mr\longrightarrow\mr^N$$
is supposed to be of class $C^1$ with a polynomial growth of its derivative, which is denoted by $b=(b_1,\dots,b_N)$.
The diffusion matrix
$$A=(A_{ij})_{i,j=1}^N:\mt^N\longrightarrow\mr^{N\times N}$$
is of class $C^\infty$, symmetric and positive semidefinite. Its square-root matrix, which is also symmetric and positive semidefinite, is denoted by $\sigma$.

Regarding the stochastic term, let $(\Omega,\mf,(\mf_t)_{t\geq0},\prst)$ be a stochastic basis with a complete, right-continuous filtration. The initial datum may be random in general, namely, we assume $u_0\in L^p(\Omega;L^p(\mt^N))$ for all $p\in[1,\infty)$. The process $W$ is a cylindrical Wiener process: $W(t)=\sum_{k\geq1}\beta_k(t) e_k$ with $(\beta_k)_{k\geq1}$ being mutually independent real-valued standard Wiener processes relative to $(\mf_t)_{t\geq0}$ and $(e_k)_{k\geq1}$ a complete orthonormal system in a separable Hilbert space $\mathfrak{U}$. In this setting, we can assume, without loss of generality, that the $\sigma$-algebra $\mf$ is countably generated and $(\mf_t)_{t\geq 0}$ is the filtration generated by the Wiener process and the initial condition.
For each $z\in L^2(\mt^N)$ we consider a mapping $\,\varPhi(z):\mathfrak{U}\rightarrow L^2(\mt^N)$ defined by $\varPhi(z)e_k=g_k(\cdot,z(\cdot))$. In particular, we suppose that $g_k\in C(\mt^N\times\mr)$ and the following conditions
\begin{equation}\label{linrust}
G^2(x,\xi)=\sum_{k\geq1}\big|g_k(x,\xi)\big|^2\leq L\big(1+|\xi|^2\big),
\end{equation}
\begin{equation}\label{skorolip}
\sum_{k\geq1}\big|g_k(x,\xi)-g_k(y,\zeta)\big|^2\leq L\big(|x-y|^2+|\xi-\zeta|h(|\xi-\zeta|)\big),
\end{equation}
are fulfilled for every $x,y\in\mt^N,\,\xi,\zeta\in\mr$, where $h$ is a continuous nondecreasing function on $\mr_+$ satisfying, for some $\alpha>0$,
\begin{equation}\label{fceh}
h(\delta)\leq C\delta^\alpha,\quad\delta<1.
\end{equation}
The conditions imposed on $\varPhi$, particularly assumption \eqref{linrust}, imply that
$$\varPhi:L^2(\mt^N)\longrightarrow L_2(\mathfrak{U};L^2(\mt^N)),$$
where $L_2(\mathfrak{U};L^2(\mt^N))$ denotes the collection of Hilbert-Schmidt operators from $\mathfrak{U}$ to $L^2(\mt^N)$. Thus, given a predictable process
$$u\in L^2(\Omega;L^2(0,T;L^2(\mt^N))),$$
the stochastic integral $t\mapsto\int_0^t\varPhi(u)\dif W$ is a well defined process taking values in $L^2(\mt^N)$ (see \cite{prato} for detailed construction).

Finally, define the auxiliary space $\mathfrak{U}_0\supset\mathfrak{U}$ via
$$\mathfrak{U}_0=\bigg\{v=\sum_{k\geq1}\alpha_k e_k;\;\sum_{k\geq1}\frac{\alpha_k^2}{k^2}<\infty\bigg\},$$
endowed with the norm
$$\|v\|^2_{\mathfrak{U}_0}=\sum_{k\geq1}\frac{\alpha_k^2}{k^2},\qquad v=\sum_{k\geq1}\alpha_k e_k.$$
Note that the embedding $\mathfrak{U}\hookrightarrow\mathfrak{U}_0$ is Hilbert-Schmidt. Moreover, trajectories of $W$ are $\prst$-a.s. in $C([0,T];\mathfrak{U}_0)$ (see \cite{prato}).

As the next step, we introduce the kinetic formulation of \eqref{rovnice} as well as the basic definitions concerning the notion of kinetic solution.

\begin{defin}[Kinetic measure]
A mapping $m$ from $\Omega$ to the set of nonnegative finite measures over $\mt^N\times[0,T]\times\mr$ is said to be a kinetic measure if
\begin{enumerate}
 \item $m$ is measurable in the following sense: for each $\psi\in C_b(\mt^N\times[0,T]\times\mr)$ the mapping $\langle m,\psi\rangle:\Omega\rightarrow\mr$ is measurable,
 \item $m$ vanishes for large $\xi$: if $B_R^c=\{\xi\in\mr;\,|\xi|\geq R\}$ then
\begin{equation}\label{infinity}
\lim_{R\rightarrow\infty}\stred\,m\big(\mt^N\times[0,T]\times B_R^c\big)=0,
\end{equation}
 \item for all $\psi\in C_b(\mt^N\times\mr),$ the process
$$t\mapsto\int_{\mt^N\times[0,t]\times\mr}\psi(x,\xi)\,\dif m(x,s,\xi)$$
is predictable.
\end{enumerate}

\end{defin}

\begin{defin}\label{formulace}
The kinetic formulation of \eqref{rovnice} is
\begin{equation}\label{kinetic}
\partial_tf+b(\xi)\cdotp\nabla f-\sum_{i,j=1}^N\partial_{x_i}\big(A_{ij}(x)\partial_{x_j}f\big)=\delta_{u=\xi}\varPhi(u)\dot{W}+\partial_\xi\bigg(m-\frac{1}{2}G^2\delta_{u=\xi}\bigg),
\end{equation}
where $m$ is a kinetic measure which consists of two components $m=n_1+n_2$. The measure $n_1$ is known and relates to the diffusion term in \eqref{rovnice}
\begin{equation*}
n_1(x,t,\xi)=(\nabla u)^*A(x)(\nabla u)\delta_{u=\xi}
\end{equation*}
whereas $n_2$ is an unknown nonnegative measure.
\end{defin}

\begin{rem}
The measure $n_1$ is called parabolic dissipative measure and gives us better regularity of solutions in the nondegeneracy zones of the diffusion matrix $A$ (cf. Definition \ref{kinsol} (iii)).
The measure $n_2$ takes account of possible singularities of solution and vanishes in the nondegenerate case.
\end{rem}

We now derive the kinetic formulation in case of a sufficiently smooth $u$ satisfying \eqref{rovnice}. It is a consequence of the It\^{o} formula. \label{odvozeni}

Let us set $\langle\ind_{u>\xi},\theta'\rangle=\int_{\mr}\ind_{u>\xi}\theta'(\xi)\dif \xi=\theta(u)$ for $\theta\in\mc_c^{\infty}(\mr)$ and apply the It\^{o} formula:
\begin{equation*}
\begin{split}
\dif \langle\ind_{u>\xi},\theta'\rangle&=\theta'(u)\big[-\diver\big(B(u)\big)\dif t+\diver\big(A(x)\nabla u\big)\dif t+\varPhi(u)\dif W(t)\big]\\
&\quad+\frac{1}{2}\theta''(u)G^2(u)\dif t.
\end{split}
\end{equation*}
Afterwards, we proceed term by term
\begin{align*}
&\theta'(u)\diver\big(B(u)\big)=\theta'(u)b(u)\cdotp\nabla u=\diver\bigg(\int^ub(\xi)\theta'(\xi)\dif\xi\bigg)=\diver\big(\langle b\ind_{u>\xi},\theta'\rangle\big)\\
\begin{split}
\theta'(u)\diver\big(A(x)\nabla u\big)&=\sum_{i,j=1}^N\partial_{x_i}\big[A_{ij}(x)\theta'(u)\partial_{x_j}u\big]-\sum_{i,j=1}^N\theta''(u)\partial_{x_i}uA_{ij}(x)\partial_{x_j}u\\
&=\sum_{i,j=1}^N\partial_{x_i}\Big(A_{ij}(x)\partial_{x_j}\langle\ind_{u>\xi},\theta'\rangle\Big)+\big\langle\partial_\xi n_1,\theta'\big\rangle\\
\end{split}\\
&\theta'(u)\varPhi(u)=\langle\delta_{u=\xi}\varPhi(u),\theta'\rangle\\
&\theta''(u)G^2(u)=\langle G^2\delta_{u=\xi},\theta''\rangle=-\big\langle \partial_\xi (G^2\delta_{u=\xi}),\theta'\big\rangle
\end{align*}

Taking $\theta(\xi)=\int^\xi\varphi$ for any test function $\varphi\in C_c^\infty(\mt^N\times[0,T]\times\mr)$ we obtain that $f=\ind_{u>\xi}$ is a distributional solution to the kinetic formulation \eqref{kinetic} with $n_2=0$. Therefore any smooth solution of \eqref{rovnice} is a kinetic solution in the sense of the following definition.

\begin{defin}[Kinetic solution]\label{kinsol}
A measurable function $u:\Omega\times\mt^N\times[0,T]\rightarrow\mr$ is said to be a kinetic solution to \eqref{rovnice} with initial datum $u_0$ provided
\begin{enumerate}
\item $(u(t);\,t\in[0,T])$ is a predictable $L^p(\mt^N)$-valued process, $p\in[1,\infty)$,
\item for all $p\in[1,\infty)$ there exists $C_p>0$ such that for $t\in[0,T]$
\begin{equation}\label{integrov}
\|u(t)\|_{L^p(\Omega\times\mt^N)}\leq C_p,
\end{equation}
\item $(\nabla u)^*A(x)(\nabla u)\in L^1(\Omega\times\mt^N\times[0,T])$,
\item there exists a kinetic measure $m$ such that $m=n_1+n_2$ as in Definition \ref{formulace} and $f=\ind_{u>\xi}$ is a weak solution of the kinetic formulation \eqref{kinetic}, i.e. $\prst\text{-a.s.}$ for all $\varphi\in C_c^\infty(\mt^N\times[0,T]\times\mr),$
\end{enumerate}
\begin{align}\label{kinet}
\begin{split}
\int_0^T&\big\langle f(t),\partial_t\varphi(t)\big\rangle\dif t+\big\langle f_0,\varphi(0)\big\rangle+\int_0^T\big\langle f(t),b(\xi)\cdotp\nabla\varphi(t)\big\rangle\dif t\\
&\qquad\qquad+\int_0^T\Big\langle f(t),\sum_{i,j=1}^N\partial_{x_j}\big(A_{ij}(x)\partial_{x_i}\varphi(t)\big)\Big\rangle\dif t\\
&=-\sum_{k\geq1}\int_0^T\int_{\mt^N}g_k\big(x,u(x,t)\big)\varphi\big(x,t,u(x,t)\big)\dif x\,\dif\beta_k(t)\\
&\qquad\qquad-\frac{1}{2}\int_0^T\int_{\mt^N}G^2\big(x,u(x,t)\big)\partial_\xi\varphi\big(x,t,u(x,t)\big)\dif x\,\dif t+m(\partial_\xi\varphi).
\end{split}
\end{align}
% \end{enumerate}

\end{defin}

We proceed by two related definitions, which will be useful especially in the proof of uniqueness.
%The concept of Young measures was developed in \cite{young} as a technical tool for describing composite limits of smooth nonlinear functions with weakly convergent sequences. A Young measure can be intuitively understood as the limit probability distribution of a weakly convergent sequence; it takes values in the set of Dirac measures if and only if the corresponding sequence converges strongly.

\begin{defin}[Young measure]
Let $(X,\lambda)$ be a finite measure space. A mapping $\nu$ from $X$ to the set of probability measures on $\mr$ is said to be a Young measure if, for all $\psi\in C_b(\mr)$, the map $z\mapsto\nu_z(\psi)$ from $X$ into $\mr$ is measurable. We say that a Young measure $\nu$ vanishes at infinity if, for all $p\geq 1$,
\begin{equation*}
\int_X\int_\mr|\xi|^p\dif\nu_z(\xi)\,\dif\lambda(z)<\infty.
\end{equation*}
\end{defin}

\begin{defin}[Kinetic function]
Let $(X,\lambda)$ be a finite measure space. A measurable function $f:X\times\mr\rightarrow[0,1]$ is said to be a kinetic function if there exists a Young measure $\nu$ on $X$ vanishing at infinity such that, for $\lambda$-a.e. $z\in X$, for all $\xi\in\mr$,
$$f(z,\xi)=\nu_z(\xi,\infty).$$
\end{defin}

% \begin{defin}[Generalized kinetic solution]
% Let $f_0:\Omega\times\mt^N\times\mr\rightarrow[0,1]$ be a kinetic function. A measurable function $f:\Omega\times\mt^N\times[0,T]\times\mr\rightarrow[0,1]$ is said to be a generalized solution to \eqref{rovnice} with initial datum $f_0$ if $(f(t);\,t\in[0,T])$ is predictable, $f(t)$ is a kinetic function for all $t\in[0,T]$ and for all $p\geq1$, there exists $C_p>0$ such that for a.e. $t\in(0,T)$
% \begin{equation}\label{odhad}
% \stred\int_{\mt^N}\int_\mr|\xi|^p\dif\nu_{x,t}(\xi)\,\dif x\leq C_p, 
% \end{equation}
% where $\nu=-\partial_\xi f$, and if there exists a kinetic measure $m$ such that $\prst\text{-a.s.}$ for all $\varphi\in C_c^\infty(\mt^N\times[0,T]\times\mr),$
% \begin{equation}\label{general}
% \begin{split}
% \int_0^T\big\langle f(t),\partial_t\varphi(t)&\big\rangle\dif t+\big\langle f_0,\varphi(0)\big\rangle+\int_0^T\big\langle f(t),b(\xi)\cdotp\nabla\varphi(t)\big\rangle\dif t\\
% &\qquad\qquad+\int_0^T\Big\langle f(t),\sum_{i,j=1}^N\partial_{x_j}\big(A_{ij}(x)\partial_{x_i}\varphi(t)\big)\Big\rangle\dif t\\
% &=-\sum_{k\geq1}\int_0^T\int_{\mt^N}\int_\mr g_k(x,\xi)\varphi(x,t,\xi)\dif\nu_{x,t}(\xi)\,\dif x\,\dif\beta_k(t)\\
% &\qquad\qquad-\frac{1}{2}\int_0^T\int_{\mt^N}\int_\mr G^2(x,\xi)\partial_\xi\varphi(x,t,\xi)\dif\nu_{x,t}(\xi)\,\dif x\,\dif t\\
% &\qquad\qquad+m(\partial_\xi\varphi).
% \end{split}
% \end{equation}
% \end{defin}

\begin{rem}
%If $f$ is a kinetic function, it holds particularly that $\,\partial_\xi f=-\nu$. Moreover,
Note, that if $f$ is a kinetic function then $\partial_\xi f=-\nu$. Similarly, let $u$ be a kinetic solution of \eqref{rovnice} and consider $f=\ind_{u >\xi}$. We have $\,\partial_\xi f=-\delta_{u=\xi}$, where $\nu=\delta_{u=\xi}$ is a Young measure on $\Omega\times\mt^N\times[0,T]$. Therefore, the expression \eqref{kinet} can be rewritten in the following form: for all $\varphi\in C_c^\infty(\mt^N\times[0,T]\times\mr)$, $\prst$-a.s.,
\begin{equation}\label{general}
\begin{split}
\int_0^T&\big\langle f(t),\partial_t\varphi(t)\big\rangle\dif t+\big\langle f_0,\varphi(0)\big\rangle+\int_0^T\big\langle f(t),b(\xi)\cdotp\nabla\varphi(t)\big\rangle\dif t\\
&\qquad\qquad+\int_0^T\Big\langle f(t),\sum_{i,j=1}^N\partial_{x_j}\big(A_{ij}(x)\partial_{x_i}\varphi(t)\big)\Big\rangle\dif t\\
&=-\sum_{k\geq1}\int_0^T\int_{\mt^N}\int_\mr g_k(x,\xi)\varphi(x,t,\xi)\dif\nu_{x,t}(\xi)\,\dif x\,\dif\beta_k(t)\\
&\qquad\qquad-\frac{1}{2}\int_0^T\int_{\mt^N}\int_\mr G^2(x,\xi)\partial_\xi\varphi(x,t,\xi)\dif\nu_{x,t}(\xi)\,\dif x\,\dif t+m(\partial_\xi\varphi).
\end{split}
\end{equation}
For a general kinetic function $f$ with corresponding Young measure $\nu$, the above formulation leads to the notion of generalized kinetic solution as used in \cite{debus}. Although this concept is not established here,
the notation will be used throughout the paper, i.e. we will often write $\nu_{x,t}(\xi)$ instead of $\delta_{u(x,t)=\xi}$.
\end{rem}

\begin{lemma}\label{kinetcomp}
Let $(X,\lambda)$ be a finite measure space such that $L^1(X)$ is separable.\footnote{According to \cite[Proposition 3.4.5]{cohn}, it is sufficient to assume that the corresponding $\sigma$-algebra is countably generated.} Let $\{f_n;\,n\in\mn\}$ be a sequence of kinetic functions on $X\times\mr$, i.e. $f_n(z,\xi)=\nu^n_z(\xi,\infty)$ where $\nu^n$ are Young measures on $X$. Suppose that, for some $p\geq1$,
$$\sup_{n\in\mn}\int_X\int_\mr|\xi|^p\dif\nu^n_z(\xi)\,\dif\lambda(z)<\infty.$$
Then there exists a kinetic function $f$ on $X\times\mr$ and a subsequence still denoted by $\{f_n;\,n\in\mn\}$ such that
$$f_n\overset{w^*}{\longrightarrow} f,\quad \text{ in }\quad L^\infty(X\times\mr)\text{-weak}^*.$$

\begin{proof}
The proof can be found in \cite[Corollary 6]{debus}.
\end{proof}

\end{lemma}

To conclude this section we state the main result of the paper.

\begin{thm}\label{main}
Let $u_0\in L^p(\Omega;L^p(\mt^N)),$ for all $p\in[1,\infty)$. Under the above assumptions, there exists a unique kinetic solution to the problem \eqref{rovnice}. Moreover, if $u_1,\,u_2$ are kinetic solutions to \eqref{rovnice} with initial data $u_{1,0}$ and $u_{2,0}$, respectively, then for all $t\in[0,T]$
$$\stred\|u_1(t)-u_2(t)\|_{L^1(\mt^N)}\leq\stred\|u_{1,0}-u_{2,0}\|_{L^1(\mt^N)}.$$
\end{thm}

\section{Uniqueness}

We begin with the question of uniqueness. Due to the following proposition, we obtain an auxiliary property of kinetic solutions, which will be useful later on in the proof of the comparison principle in Theorem \ref{uniqueness}. Since the proof is very similar to \cite[Proposition 8]{debus}, it will be left to the reader.

\begin{prop}[Left and right weak limits]\label{limits}
Let $u$ be a kinetic solution to \eqref{rovnice}. Then $f=\ind_{u>\xi}$ admits almost surely left and right limits at all points $t^*\in[0,T]$ in the sense of distributions over $\mt^N\times\mr$, i.e. for all $t^*\in[0,T]$ there exist some kinetic functions $f^{*,\pm}$ on $\,\Omega\times\mt^N\times\mr$ such that
$$\big\langle f(t^*\pm\,\varepsilon),\psi\big\rangle\longrightarrow\big\langle f^{*,\pm},\psi\big\rangle,\quad\varepsilon\downarrow 0,\quad\forall\psi\in\mc^\infty_c(\mt^N\times\mr)\quad\prst\text{-a.s.}.$$
Moreover, $f^{*,+}=f^{*,-}=f(t^*)$ almost surely for all $t^*\in[0,T]$ except for some countable set, i.e.
$$\big\langle f^{*,+},\psi\big\rangle=\big\langle f^{*,-},\psi\big\rangle=\big\langle f(t^*),\psi\big\rangle\quad\forall\psi\in\mc^\infty_c(\mt^N\times\mr)\quad\prst\text{-a.s.}$$

\end{prop}

As the next step towards the proof of the comparison principle, we need a technical proposition relating two kinetic solutions of \eqref{rovnice}.
We will also use the following notation: if $f:X\times\mr\rightarrow[0,1]$ is a kinetic function, we denote by $\bar{f}$ the conjugate function $\bar{f}=1-f$. We define the function $f^\pm$ by $f^\pm(t^*)=f^{*,\pm},\,t^*\in[0,T]$.

\begin{prop}\label{propdubling}
Let $u_1,\,u_2$ be two kinetic solutions to \eqref{rovnice} and denote $f_1=\ind_{u_1>\xi},\,f_2=\ind_{u_2>\xi}$. Then for $t\in[0,T]$ and any nonnegative functions $\varrho\in\mc^\infty(\mt^N),\,\psi\in\mc^\infty_c(\mr)$ we have
\begin{equation}\label{doubling}
\begin{split}
\stred&\int_{(\mt^N)^2}\int_{\mr^2}\varrho(x-y)\psi(\xi-\zeta)f_1^{\pm}(x,t,\xi)\bar{f}_2^{\pm}(y,t,\zeta)\,\dif\xi\,\dif\zeta\,\dif x\,\dif y\\
&\leq\stred\int_{(\mt^N)^2}\int_{\mr^2}\varrho(x-y)\psi(\xi-\zeta)f_{1,0}(x,\xi)\bar{f}_{2,0}(y,\zeta)\,\dif\xi\,\dif\zeta\,\dif x\,\dif y+\mathrm{I}+\mathrm{J}+\mathrm{K},
\end{split}
\end{equation}
where
$$\mathrm{I}=\,\stred\int_0^t\int_{(\mt^N)^2}\int_{\mr^2}f_1\bar{f}
_2\big(b(\xi)-b(\zeta)\big)\cdotp\nabla_x\alpha(x,\xi,y,\zeta)\,\dif\xi\,
\dif\zeta\,\dif x\, \dif y\,\dif s,$$
\begin{equation*}
\begin{split}
\mathrm{J}=&\,\stred\int_0^t\int_{(\mt^N)^2}\int_{\mr^2}f_1\bar{f}_2\sum_{i,j=1}^N\partial_{y_j}\big(A_{ij}(y)\partial_{y_i}\alpha\big)\,\dif\xi\,\dif\zeta\,\dif x\, \dif y\,\dif s\\
&\qquad+\stred\int_0^t\int_{(\mt^N)^2}\int_{\mr^2}f_1\bar{f}_2\sum_{i,j=1}^N\partial_{x_j}\big(A_{ij}(x)\partial_{x_i}\alpha\big)\,\dif\xi\,\dif\zeta\,\dif x\, \dif y\,\dif s\\
&\qquad-\stred\int_0^t\int_{(\mt^N)^2}\int_{\mr^2}\alpha(x,\xi,y,\zeta)\,
\dif\nu^ { 1 , + } _ { x , s }
(\xi)\,\dif x\,\dif n_{2,1}(y,s,\zeta)\\
&\qquad-\stred\int_0^t\int_{(\mt^N)^2}\int_{\mr^2}\alpha(x,\xi,y,\zeta)\,
\dif\nu^{2,-}_{y,s}
(\zeta)\,\dif y\,\dif n_{1,1}(x,s,\xi),
\end{split}
\end{equation*}
$$\mathrm{K}=\frac{1}{2}\stred\int_0^t\!\!\int_{(\mt^N)^2}\!\!\int_{\mr^2}
\!\!\alpha(x,\xi,y,\zeta)\!\sum_{k\geq1}\!\big|g_k(x,\xi)-g_k(y,
\zeta)\big|^2\dif\nu^1_{x,s}(\xi)\dif\nu^2_{y,s}(\zeta)\dif x\,\dif y\,\dif s,$$
and the function $\alpha$ is defined as $\alpha(x,\xi,y,\zeta)=\varrho(x-y)\psi(\xi-\zeta)$.

\begin{proof}
Let us denote by $\langle\!\langle\cdot,\cdot\rangle\!\rangle$ the scalar product in $L^2(\mt^N_x\times\mt^N_y\times\mr_\xi\times\mr_\zeta)$. In order to prove the statement in the case of $f^+_1,\,\bar{f}^+_2$, we employ similar calculations as in \cite[Proposition 9]{debus} to obtain

\begin{equation}\label{final}
\begin{split}
\stred\bl\!\bl f_1^+(t)&\bar{f}_2^+(t),\alpha\bR\!\bR=\stred\bl\!\bl f_{1,0}\bar{f}_{2,0},\alpha\bR\!\bR\\
&+\stred\int_0^t\int_{(\mt^N)^2}\int_{\mr^2}f_1\bar{f}_2\big(b(\xi)-b(\zeta)\big)\cdotp\nabla_x\alpha\,\dif\xi\,\dif\zeta\,\dif x\, \dif y\,\dif s\\
&+\stred\int_0^t\int_{(\mt^N)^2}\int_{\mr^2}f_1\bar{f}_2\sum_{i,j=1}^N\partial_{y_j}\big(A_{ij}(y)\partial_{y_i}\alpha\big)\,\dif\xi\,\dif\zeta\,\dif x\, \dif y\,\dif s\\
&+\stred\int_0^t\int_{(\mt^N)^2}\int_{\mr^2}f_1\bar{f}_2\sum_{i,j=1}^N\partial_{x_j}\big(A_{ij}(x)\partial_{x_i}\alpha\big)\,\dif\xi\,\dif\zeta\,\dif x\, \dif y\,\dif s\\
&+\frac{1}{2}\stred\int_0^t\int_{(\mt^N)^2}\int_{\mr^2}\bar{f}_2\partial_\xi\alpha\, G^2_1\,\dif\nu_{x,s}^1(\xi)\,\dif\zeta\,\dif y\,\dif x\,\dif s\\
&-\frac{1}{2}\stred\int_0^t\int_{(\mt^N)^2}\int_{\mr^2}f_1\partial_\zeta\alpha\, G^2_2\,\dif\nu_{y,s}^2(\zeta)\,\dif\xi\,\dif y\,\dif x\,\dif s\\
&-\stred\int_0^t\int_{(\mt^N)^2}\int_{\mr^2}G_{1,2}\alpha\,\dif\nu^1_{x,s}(\xi)\,\dif\nu^2_{y,s}(\zeta)\,\dif x\,\dif y\,\dif s\\
&-\stred\int_0^t\int_{(\mt^N)^2}\int_{\mr^2}\bar{f}_2^-\partial_\xi\alpha\,\dif m_1(x,s,\xi)\,\dif\zeta\,\dif y\\
&+\stred\int_0^t\int_{(\mt^N)^2}\int_{\mr^2}f_1^+\partial_\zeta\alpha\,\dif m_2(y,s,\zeta)\,\dif\xi\,\dif x.
\end{split}
\end{equation}
In particular, since $\alpha\geq 0$, the last term in \eqref{final} satisfies
\begin{equation*}
\begin{split}
\stred\int_0^t\int_{(\mt^N)^2}\int_{\mr^2}&f_1^+\partial_\zeta\alpha\,\dif m_2(y,s,\zeta)\,\dif\xi\,\dif x\\
%&=-\stred\int_0^t\int_{(\mt^N)^2}\int_{\mr^2}f_1^+\partial_\xi\alpha\,\dif m_2(y,s,\zeta)\,\dif\xi\,\dif x\\
%&=-\stred\int_0^t\int_{(\mt^N)^2}\int_{\mr^2}\alpha\,\dif\nu^{1,+}_{x,s}(\xi)\,\dif x\,\dif m_2(y,s,\zeta)\\
&=-\stred\int_0^t\int_{(\mt^N)^2}\int_{\mr^2}\alpha\,\dif\nu^{1,+}_{x,s}(\xi)\,\dif x\,\dif n_{2,1}(y,s,\zeta)\\
&\qquad-\stred\int_0^t\int_{(\mt^N)^2}\int_{\mr^2}\alpha\,\dif\nu^{1,+}_{x,s}
(\xi)\,\dif x\,\dif n_{2,2}(y,s,\zeta)\\
&\leq-\stred\int_0^t\int_{(\mt^N)^2}\int_{\mr^2}\alpha\,\dif\nu^{1,+}_{x,s}
(\xi)\,\dif x\,\dif n_{2,1}(y,s,\zeta)
\end{split}
\end{equation*}
and by symmetry
\begin{equation*}
\begin{split}
-\stred\int_0^t\int_{(\mt^N)^2}\int_{\mr^2}&\bar{f}_2^-\partial_\xi\alpha\,\dif m_1(x,s,\xi)\,\dif\zeta\,\dif y\\
&\leq-\stred\int_0^t\int_{(\mt^N)^2}\int_{\mr^2}\alpha\,\dif\nu^{2,-}_{y,s}
(\zeta)\,\dif y\,\dif n_{1,1}(x,s,\xi).
\end{split}
\end{equation*}
Thus, the desired estimate \eqref{doubling} follows.
In the case of $f_1^-,\,\bar{f}_2^-$ we take $t_n\uparrow t$, write \eqref{doubling} for $f_1^+(t_n),\,\bar{f}_2^+(t_n)$ and let $n\rightarrow\infty$.
\end{proof}

\end{prop}

\begin{thm}[Comparison principle]\label{uniqueness}
Let $u$ be a kinetic solution to \eqref{rovnice}. Then $u$ has left and right limits at any point in the sense of $\,L^p(\mt^N),\,p\in[1,\infty),$ and, for all $\,t\in[0,T]$, $f^\pm(x,t,\xi)=\ind_{u^\pm(x,t)>\xi}$ a.s., for a.e. $(x,\xi)$. Moreover, if $\,u_1,u_2$ are kinetic solutions to \eqref{rovnice} with initial data $u_{1,0}$ and $u_{2,0}$, respectively, then for all $t\in[0,T]$
\begin{equation}\label{comparison}
\stred\|u_1^\pm(t)-u^\pm_2(t)\|_{L^1(\mt^N)}\leq \stred\|u_{1,0}-u_{2,0}\|_{L^1(\mt^N)}.
\end{equation}

\begin{proof}
Denote $f_1=\ind_{u_1>\xi},\,f_2=\ind_{u_2>\xi}$. Let $(\varrho_\tau),\,(\psi_\delta)$ be approximations to the identity on $\mt^N\text{ and }\mr$, respectively, i.e. let $\varrho\in\mc^\infty(\mt^N),\,\psi\in\mc^\infty_c(\mr)$ be nonnegative symmetric functions satisfying
$\int_{\mt^N}\varrho=1,\,\int_\mr\psi=1$
and $\supt\varrho\subset B(0,1/2),\,\supt\psi\subset (-1,1)$. We define
$$\varrho_\tau(x)=\frac{1}{\tau^N}\,\varrho\Big(\frac{x}{\tau}\Big),\qquad\psi_\delta(\xi)=\frac{1}{\delta}\,\psi\Big(\frac{\xi}{\delta}\Big).$$
Then
\begin{equation*}
\begin{split}
&\stred\int_{\mt^N}\int_\mr f_1^\pm(x,t,\xi)\bar{f}_2^\pm(x,t,\xi)\,\dif\xi\,\dif x\\
&=\stred\int_{(\mt^N)^2}\int_{\mr^2}\varrho_\tau(x-y)\psi_\delta(\xi-\zeta)f_1^\pm(x,t,\xi)\bar{f}_2^\pm(y,t,\zeta)\,\dif\xi\,\dif \zeta\,\dif x \,\dif y+\eta_t(\tau,\delta),
\end{split}
\end{equation*}
where $\lim_{\tau,\delta\rightarrow 0}\eta_t(\tau,\delta)=0$. With regard to Proposition \ref{propdubling}, we need to find suitable bounds for terms $\mathrm{I},\,\mathrm{J},\,\mathrm{K}$.

Since $b$ has at most polynomial growth, there exist $C>0,\,p>1$ such that
$$\big|b(\xi)-b(\zeta)\big|\leq\varGamma(\xi,\zeta)|\xi-\zeta|,\qquad\varGamma(\xi,\zeta)\leq C\big(1+|\xi|^{p-1}+|\zeta|^{p-1}\big).$$
Hence
$$|\mathrm{I}|\leq\stred\int_0^t\int_{(\mt^N)^2}\int_{\mr^2}f_1\bar{f}_2\varGamma(\xi,\zeta)|\xi-\zeta|\psi_\delta(\xi-\zeta)\,\dif\xi\,\dif \zeta\,\big|\nabla_x\varrho_\tau(x-y)\big|\dif x\,\dif y\,\dif s.$$
As the next step we apply integration by parts with respect to $\zeta,\,\xi$. Focusing only on the relevant integrals we get
\begin{equation}\label{nn}
\begin{split}
\int_\mr f_1(\xi)\int_\mr&\bar{f}_2(\zeta)\varGamma(\xi,\zeta)|\xi-\zeta|\psi_\delta(\xi-\zeta)\dif\zeta\,\dif\xi\\
&=\int_\mr f_1(\xi)\int_\mr\varGamma(\xi,\zeta')|\xi-\zeta'|\psi_\delta(\xi-\zeta')\dif\zeta'\,\dif\xi\\
&\qquad-\int_{\mr^2} f_1(\xi)\int_{-\infty}^\zeta\varGamma(\xi,\zeta')|\xi-\zeta'|\psi_\delta(\xi-\zeta')\dif\zeta'\,\dif\xi\,\dif\nu^2_{y,s}(\zeta)\\
&=\int_{\mr^2}f_1(\xi)\int^{\infty}_\zeta\varGamma(\xi,\zeta')|\xi-\zeta'|\psi_\delta(\xi-\zeta')\dif\zeta'\,\dif\xi\,\dif\nu^2_{y,s}(\zeta)\\
&=\int_{\mr^2}\varUpsilon(\xi,\zeta)\dif\nu^1_{x,s}(\xi)\dif\nu^2_{y,s}(\zeta)
\end{split}
\end{equation}
where
$$\varUpsilon(\xi,\zeta)=\int_{-\infty}^\xi\int_\zeta^\infty\varGamma(\xi',\zeta')|\xi'-\zeta'|\psi_\delta(\xi'-\zeta')\dif\zeta'\,\dif\xi'.$$
Therefore, we find
$$|\mathrm{I}|\leq\stred\int_0^t\int_{(\mt^N)^2}\int_{\mr^2}\varUpsilon(\xi,\zeta)\,\dif\nu^1_{x,s}(\xi)\dif\nu^2_{y,s}(\zeta)\,\big|\nabla_x\varrho_\tau(x-y)\big|\dif x\,\dif y\,\dif s.$$
The function $\varUpsilon$ can be estimated using the substitution $\xi''=\xi'-\zeta'$
\begin{equation*}
\begin{split}
\varUpsilon(\xi,\zeta)=&\int_\zeta^\infty\int_{|\xi''|<\delta,\,\xi''<\xi-\zeta'}\varGamma(\xi''+\zeta',\zeta')|\xi''|\psi_\delta(\xi'')\,\dif\xi''\,\dif\zeta'\\
&\leq C \delta \int_\zeta^{\xi+\delta}\max_{|\xi''|<\delta,\,\xi''<\xi-\zeta'}\varGamma(\xi''+\zeta',\zeta')\,\dif\zeta'\\
&\leq C \delta \int_\zeta^{\xi+\delta}\big(1+|\xi|^{p-1}+|\zeta'|^{p-1}\big)\,\dif\zeta'\\
&\leq C\delta\big(1+|\xi|^{p}+|\zeta'|^{p}\big)
\end{split}
\end{equation*}
hence, since $\nu^1,\,\nu^2$ vanish at infinity,
$$|\mathrm{I}|\leq C t\delta\int_{\mt^N}\big|\nabla_x\varrho_\tau(x)\big|\,\dif x\leq C t\delta\tau^{-1}.$$

We recall that $f_1=\ind_{u_1(x,t)>\xi},\,f_2=\ind_{u_2(y,t)>\zeta}$
and $$\partial_\xi f_1=-\nu^1=-\delta_{u_1(x,t)=\xi},\qquad\partial_\zeta
f_2=-\nu^2=-\delta_{u_2(y,t)=\zeta}.$$

The first term in $\mathrm{J}$ can be rewritten in the following manner using
integration by parts (and considering only relevant integrals)
\begin{equation*}
\begin{split}
\int_{\mt^N}\bar{f}_2\int_{\mt^N}
f_1&\,\partial_{x_j}\big(A_{ij}(x)\partial_{x_i}\varrho_\tau(x-y)\big)\dif
x\,\dif y\\
&=-\int_{\mt^N}\bar{f}_2(y,s,\zeta)\int_{\mt^N}\partial_{x_j}f_1(x,s,
\xi)A_{ij}(x)\partial_{x_i}\varrho_\tau(x-y)\dif x\,\dif y\\
&=\int_{(\mt^N)^2}\bar{f}_2(y,s,\zeta)\partial_{x_j}f_1(x,s,
\xi)A_{ij}(x)\partial_{y_i}\varrho_\tau(x-y)\dif
x\,\dif y\\
&=-\int_{(\mt^N)^2}\partial_{x_j}f_1(x,s,\xi)A_{ij}(x)\partial_{y_i}\bar{f}_2(y,s,
\zeta)\varrho_\tau(x-y)\dif x\,\dif y,
\end{split}
\end{equation*}
similarly
\begin{equation*}
\begin{split}
\int_{\mt^N}f_1\int_{\mt^N}
\bar{f}_2&\,\partial_{y_j}\big(A_{ij}(y)\partial_{y_i}\varrho_\tau(x-y)\big)\dif
y\,\dif
x\\
&=-\int_{(\mt^N)^2}\partial_{x_i}f_1(x,s,\xi)A_{ij}(y)\partial_{y_j}\bar{f}_2(y,s,
\zeta)\varrho_\tau(x-y)\dif x\,\dif y.
\end{split}
\end{equation*}
% Note, that the following part of the proof is incomplete because we use nonlinear manipulations for weak solutions. The rigorous proof requires first the regularization of equation \eqref{kinetic} by convolution, afterwards it follows the formal calculations performed here.
Using regularization by convolutions, it is possible to show that the following holds in the sense of distributions over $\mt^N\times\mr$
$$\partial_{x_i}f_1=\partial_{x_i}
u_1\delta_{u_1(x,s)=\xi},\qquad\partial_{y_i}\bar{f}_2=-\partial_{y_i}
u_2\delta_{u_2(y,s)=\zeta}.$$
Hence
\begin{equation*}
\begin{split}
\mathrm{J}&=\,\stred\int_0^t\int_{(\mt^N)^2}(\nabla_x
u_1)^*\!A(x)(\nabla_y
u_2)\varrho_\tau(x-y)\psi_\delta\big(u_1(x,s)-u_2(y,s)\big)\,\dif x\,\dif y\,\dif s\\
&+\stred\int_0^t\int_{(\mt^N)^2}(\nabla_x
u_1)^*\!A(y)(\nabla_y
u_2)\varrho_\tau(x-y)\psi_\delta\big(u_1(x,s)-u_2(y,s)\big)\,\dif x\,\dif y\,\dif s\\
&-\stred\int_0^t\int_{(\mt^N)^2}(\nabla_y
u_2)^*\!A(y)(\nabla_y
u_2)\varrho_\tau(x-y)\psi_\delta\big(u_1(x,s)-u_2(y,s)\big)\,\dif x\,\dif y\,\dif s\\
&-\stred\int_0^t\int_{(\mt^N)^2}(\nabla_x
u_1)^*\!A(x)(\nabla_x
u_1)\varrho_\tau(x-y)\psi_\delta\big(u_1(x,s)-u_2(y,s)\big)\,\dif x\,\dif y\,\dif s.
\end{split}
\end{equation*}
Let us define
$$\Theta_\delta(\xi)=\int_{-\infty}^\xi\psi_\delta(\zeta)\,\dif \zeta.$$
Then we have $\mathrm{J}=\mathrm{J}_1+\mathrm{J}_2+\mathrm{J}_3$ with
\begin{equation*}
\begin{split}
\mathrm{J}_1=&-\stred\!\int_0^t\int_{(\mt^N)^2}\!\!(\nabla_x u_1)^*\!\sigma(x)\sigma(x)(\nabla\varrho_\tau)(x-y)\Theta_\delta\big(u_1(x,s)-u_2(y,s)\big)\dif x\dif y\dif s,\\
\mathrm{J}_2=&-\stred\!\int_0^t\int_{(\mt^N)^2}\!\!(\nabla_y u_2)^*\!\sigma(y)\sigma(y)(\nabla\varrho_\tau)(x-y)\Theta_\delta\big(u_1(x,s)-u_2(y,s)\big)\dif x\dif y\dif s,\\
\mathrm{J}_3=&-\stred\!\int_0^t\int_{(\mt^N)^2}\!\big[|\sigma(x)\nabla_x u_1|^2+|\sigma(y)\nabla_y u_2|^2\big]\varrho_\tau(x-y)\\
&\qquad\qquad\quad\times\psi_\delta\big(u_1(x,s)-u_2(y,s)\big)\,\dif x\,\dif y\,\dif s.
\end{split}
\end{equation*}
Let
$$\mathrm{H}=\stred\!\int_0^t\int_{(\mt^N)^2}\!(\nabla_x u_1)^*\sigma(x)\sigma(y)(\nabla_y u_2)\varrho_\tau(x-y)\psi_\delta\big(u_1(x,s)-u_2(y,s)\big)\,\dif x\,\dif y\,\dif s.$$
In order to prove that $\mathrm{J}$ is nonpositive for $\tau$ small enough, it is sufficient to show
$\mathrm{J}_1=\mathrm{H}+o(1),$ $\mathrm{J}_2=\mathrm{H}+o(1)$, where $o(1)\rightarrow 0$ as $\tau\rightarrow 0$ uniformly in $\delta$. Indeed, we then obtain
\begin{equation*}
\begin{split}
\mathrm{J}=&-\stred\!\int_0^t\int_{(\mt^N)^2}\!\!\big|\sigma(x)\nabla_x u_1-\sigma(y)\nabla_y u_2\big|^2\!\varrho_\tau(x-y)\\
&\qquad\qquad\times\psi_\delta\big(u_1(x,s)-u_2(y,s)\big)\,\dif x\,\dif y\,\dif s+o(1)\leq o(1).
\end{split}
\end{equation*}
We will prove the claim for $\mathrm{J}_1$ only since the proof for $\mathrm{J}_2$ is analogous. Define
$$g(x,y,s)=(\nabla_x u_1)^*\sigma(x)\Theta_\delta\big(u_1(x,s)-u_2(y,s)\big).$$
Here, we make use of the assumption (iii) in Definition \ref{kinsol}. Observe, that it gives us some regularity of the solution in the nondegeneracy zones of the diffusion matrix $A$ and hence $g\in L^2(\Omega\times\mt^N_x\times\mt^N_y\times[0,T])$.
It holds
\begin{equation*}
\begin{split}
\mathrm{J}_1&=-\stred\int_0^t\int_{(\mt^N)^2}g(x,y,s)\Big(\sigma(x)-\sigma(y)\Big)(\nabla\varrho_\tau)(x-y)\,\dif x\,\dif y\,\dif s\\
&\qquad-\stred\int_0^t\int_{(\mt^N)^2}g(x,y,s)\sigma(y)(\nabla\varrho_\tau)(x-y)\,\dif x\,\dif y\,\dif s,\\
\mathrm{H}&=\stred\int_0^t\int_{(\mt^N)^2}g(x,y,s)\diver_y\Big(\sigma(y)\varrho_\tau(x-y)\Big)\,\dif x\,\dif y\,\dif s\\
&=\stred\int_0^t\int_{(\mt^N)^2}g(x,y,s)\diver\big(\sigma(y)\big)\varrho_\tau(x-y)\,\dif x\,\dif y\,\dif s\\
&\qquad-\stred\int_0^t\int_{(\mt^N)^2}g(x,y,s)\sigma(y)(\nabla\varrho_\tau)(x-y)\,\dif x\,\dif y\,\dif s,
\end{split}
\end{equation*}
where divergence is applied row-wise to a matrix-valued function. Therefore, it is enough to show that the first terms in $\mathrm{J}_1$ and $\mathrm{H}$ have the same limit value if $\tau\rightarrow 0$.
For $\mathrm{H}$, we obtain easily
\begin{equation*}
\begin{split}
\stred\int_0^t\int_{(\mt^N)^2}g(x,y,s)\diver&\big(\sigma(y)\big)\varrho_\tau(x-y)\,\dif x\,\dif y\,\dif s\\
&\quad\longrightarrow\,\stred\int_0^t\int_{\mt^N}g(y,y,s)\diver\big(\sigma(y)\big)\,\dif y\,\dif s\\
\end{split}
\end{equation*}
so it remains to verify
\begin{equation*}
\begin{split}
-\stred\int_0^t\int_{(\mt^N)^2}g(x,y,s)\Big(\sigma(x)&-\sigma(y)\Big)(\nabla\varrho_\tau)(x-y)\,\dif x\,\dif y\,\dif s\\
&\longrightarrow\stred\int_0^t\int_{\mt^N}g(y,y,s)\diver\big(\sigma(y)\big)\,\dif y\,\dif s.
\end{split}
\end{equation*}
We will use similar arguments as in the commutation lemma of DiPerna and Lions (see \cite[Lemma II.1]{diperna}). Let us denote by $g^i$ the $i^{th}$ element of $g$ and by $\sigma^i$ the $i^{th}$ row of $\sigma$. Since $\tau|\nabla\varrho_\tau|(\cdot)\leq C \varrho_{2\tau}(\cdot)$ with a constant independent of $\tau$, we obtain the following estimate
\begin{equation*}
\begin{split}
\stred\int_0^t&\int_{\mt^N}\bigg|\int_{\mt^N}g^i(x,y,s)\Big(\sigma^i(x)-\sigma^i(y)\Big)(\nabla\varrho_\tau)(x-y)\,\dif x\bigg|\,\dif y\,\dif s\\
&\leq C\esssup_{\substack{x',y'\in\mt^N\\|x'-y'|\leq\tau}}\bigg|\frac{\sigma^i(x')-\sigma^i(y')}{\tau}\bigg|\,\stred \int_0^T\int_{(\mt^N)^2}\big|g^i(x,y,s)\big|\varrho_{2\tau}(x-y)\,\dif x\,\dif y\,\dif s.
\end{split}
\end{equation*}
Note that according to \cite{freidlin}, \cite{phillips}, the square-root matrix of $A$ is Lipschitz continuous and therefore the essential supremum can be estimated by a constant independent of $\tau$.
%\begin{equation*}
%\begin{split}
%\esssup_{\substack{x',y'\in\mt^N\\|x'-y'|\leq\tau}}\bigg|\frac{\sigma^i(x')-\sigma^i(y')}{\tau}\bigg|&
%%\leq\esssup_{\substack{x,y\in\mt^N\\|x-y|\leq\tau}}\bigg|\int_0^1\totdif\sigma^i\big(y+r(x-y)\big)\frac{x-y}{\tau}\,\dif r\bigg|\\
%%&\leq\esssup_{\substack{|z|\leq1}}\int_0^1\big|\totdif \sigma^i(y+r\tau z)z\big|\,\dif r
%\leq C\,\|\sigma^i\|_{W^{1,\infty}(\mt^N)}
%\end{split}
%\end{equation*}
Next
\begin{equation*}
\begin{split}
\stred \int_0^T\int_{(\mt^N)^2}&\big|g^i(x,y,s)\big|\varrho_{2\tau}(x-y)\,\dif x\,\dif y\,\dif s\\
&\leq\bigg(\stred\int_0^T\int_{(\mt^N)^2}\big|g^i(x,y,s)\big|^2\varrho_{2\tau}(x-y)\,\dif x\,\dif y\,\dif s\bigg)^\frac{1}{2}\\
&\qquad\qquad\times\bigg(\int_{(\mt^N)^2}\varrho_{2\tau}(x-y)\,\dif x\,\dif y\bigg)^\frac{1}{2}\\
&\leq\bigg(\stred\int_0^T\int_{\mt^N}\big|(\nabla_x u_1)^*\sigma(x)\big|^2\int_{\mt^N}\varrho_{2\tau}(x-y)\,\dif y\,\dif x\,\dif s\bigg)^\frac{1}{2}\\
&\leq\big\|(\nabla_x u_1)^*\sigma(x)\big\|_{L^2(\Omega\times\mt^N\times[0,T])}.
\end{split}
\end{equation*}
So we get an estimate which is independent of $\tau$ and $\delta$.
%It holds
%\begin{equation*}
%\begin{split}
%&-\stred\int_0^t\int_{(\mt^N)^2}g^i(x,y)\Big(\sigma^i(x)-\sigma^i(y)\Big)(\nabla\varrho_\tau)(x-y)\,\dif x\,\dif y\,\dif s\\
%&\qquad=-\frac{1}{\tau^{N+1}}\stred\int_0^t\int_{(\mt^N)^2}\int_0^1g^i(x,y)\,\totdif\sigma^i\big(y+r(x-y)\big)(x-y)\,\cdotp(\nabla\varrho)\Big(\frac{x-y}{\tau}\Big)\,\dif r\,\dif x\,\dif y\,\dif %s\\
%&\qquad=-\stred\int_0^t\int_{(\mt^N)^2}\int_0^1g^i(y+\tau z,y)\,\totdif\sigma^i(y+r\tau z)z\,\cdotp(\nabla\varrho)(z)\,\dif r\,\dif z\,\dif y\,\dif s.
%\end{split}
%\end{equation*}
It is sufficient to consider the case when $g^i$ and $\sigma^i$ are smooth. The general case follows by density argument from the above bound. It holds
\begin{equation*}
\begin{split}
&-\stred\int_0^t\int_{(\mt^N)^2}g^i(x,y)\Big(\sigma^i(x)-\sigma^i(y)\Big)(\nabla\varrho_\tau)(x-y)\,\dif x\,\dif y\,\dif s\\
&\qquad=-\frac{1}{\tau^{N+1}}\stred\int_0^t\int_{(\mt^N)^2}\int_0^1g^i(x,y)\,\totdif\sigma^i\big(y+r(x-y)\big)(x-y)\\
&\qquad\qquad\qquad\qquad\qquad\qquad\cdotp(\nabla\varrho)\Big(\frac{x-y}{\tau}\Big)\,\dif r\,\dif x\,\dif y\,\dif s\\
&\qquad=-\stred\int_0^t\int_{(\mt^N)^2}\int_0^1g^i(y+\tau z,y)\,\totdif\sigma^i(y+r\tau z)z\,\cdotp(\nabla\varrho)(z)\,\dif r\,\dif z\,\dif y\,\dif s\\
&\longrightarrow\, -\stred\int_0^t\int_{(\mt^N)^2}g^i(y,y)\,\totdif\sigma^i(y)z\,\cdotp(\nabla\varrho)(z)\,\dif z\,\dif y\,\dif s.
\end{split}
\end{equation*}
Integration by parts now yields
$$\int_{\mt^N}z_i\partial_{z_j}\varrho(z)\,\dif z=-\delta_{ij},\qquad i,j\in\{1,\dots,N\},$$
hence
\begin{equation*}
\begin{split}
-\stred\!\int_0^t\!\int_{(\mt^N)^2}\!\!g^i(y,y)\totdif\sigma^i(y)z\,\cdotp(\nabla\varrho)(z)\dif z\dif y\dif s=\stred\!\int_0^t\!\int_{\mt^N}\!\!g^i(y,y)\diver\big(\sigma^i(y)\big)\dif y\dif s
\end{split}
\end{equation*}
and we deduce finally that $\mathrm{J}$ is nonpositive.

The last term $\mathrm{K}$ is, due to \eqref{skorolip}, bounded as follows
\begin{equation*}
\begin{split}
\mathrm{K}&\leq\frac{L}{2}\stred\int_0^t\int_{(\mt^N)^2}\!\varrho_\tau(x-y)|x-y|^2\!\int_{\mr^2}\psi_\delta(\xi-\zeta)\,\dif\nu^1_{x,s}(\xi)\,\dif\nu^2_{y,s}(\zeta)\,\dif x\,\dif y\,\dif s\\
&\;+\frac{L}{2}\stred\!\int_0^t\!\int_{(\mt^N)^2}\!\!\varrho_\tau(x-y)\!\int_{\mr^2}\!\psi_\delta(\xi-\zeta)|\xi-\zeta|h(|\xi-\zeta|)\dif\nu^1_{x,s}(\xi)\dif\nu^2_{y,s}(\zeta)\dif x\dif y\dif s\\
&\leq\frac{Lt}{2\delta}\int_{(\mt^N)^2}|x-y|^2\varrho_\tau(x-y)\,\dif x\,\dif y+\frac{LtC_\psi h(\delta)}{2}\int_{(\mt^N)^2}\varrho_\tau(x-y)\,\dif x\,\dif y\\
&\leq\frac{Lt}{2}\delta^{-1}\tau^2+\frac{LtC_\psi h(\delta)}{2},
\end{split}
\end{equation*}
where $C_\psi=\sup_{\xi\in\mr}|\xi\psi(\xi)|$. Finally, we deduce for all $t\in[0,T]$
\begin{equation*}
\begin{split}
\stred\int_{\mt^N}\int_\mr f_1^{\pm}(x,t,\xi)&\bar{f}_2^{\pm}(x,t,\xi)\,\dif\xi\,\dif x\leq\stred\int_{\mt^N}\int_{\mr}f_{1,0}(x,\xi)\bar{f}_{2,0}(x,\xi)\,\dif\xi\,\dif x\\
&+Ct\big(\delta\tau^{-1}+\delta^{-1}\tau^2+h(\delta)\big)+\eta_t(\tau,\delta)+\eta_0(\tau,\delta).
\end{split}
\end{equation*}
Taking $\delta=\tau^{4/3}$ and letting $\tau\rightarrow0$ yields
$$\stred\int_{\mt^N}\int_\mr f_1^{\pm}(t)\bar{f}_2^{\pm}(t)\,\dif\xi\,\dif x\leq\stred\int_{\mt^N}\int_\mr f_{1,0}\bar{f}_{2,0}\,\dif\xi\,\dif x.$$
Let us consider now $f_1=f_2=f$. Since $f_0=\ind_{u_0>\xi}$ we have the identity $f_0\bar{f}_0=0$ and therefore $f^\pm(1-f^\pm)=0$ a.e. $(\omega,x,\xi)$ and for all $t$. The fact that $f^\pm$ is a kinetic function hence implies that there exist $u^\pm:\Omega\times\mt^N\times[0,T]\rightarrow\mr$ such that $f^\pm=\ind_{u^\pm>\xi}$ for almost every $(\omega,x,\xi)$ and all $t$.
Furthermore, it follows from Proposition \ref{limits} that $u^+=u^-=u$ except for a countable set of $t$. Since
$$\int_\mr \ind_{u^\pm_1>\xi}\overline{\ind_{u^\pm_2>\xi}}\,\dif\xi=(u^\pm_1-u^\pm_2)^+$$
we have the comparison property
$$\stred\big\|\big(u_1^{\pm}(t)-u_2^{\pm}(t)\big)^+\big\|_{L^1(\mt^N)}\leq\stred\big\|(u_{1,0}-u_{2,0})^+\big\|_{L^1(\mt^N)}.$$
It remains to show that $u^\pm$ are also left and right limits of $u$ in the sense of $L^p(\mt^N)$. But this is a consequence of the following: for $p\geq 1$ and $s,t\in[0,T],s>t$
\begin{equation*}
\begin{split}
\int_{\mt^N}\big|u(x,s)-u^+(x,t)\big|^p\dif x&=\int_{\mt^N}\int_\mr\big(\ind_{u(x,s)>\xi}-\ind_{u^+(x,t)>\xi}\big)\frac{\dif|\xi|^p}{\dif \xi}\,\dif\xi\,\dif x\\
&=\int_{\mt^N}\int_\mr\big(f(x,s,\xi)-f^+(x,t,\xi)\big)\frac{\dif|\xi|^p}{\dif \xi}\,\dif\xi\,\dif x
\end{split}
\end{equation*}
and the claim follows from the weak*-convergence of $f(s)\rightarrow f^+(t)$ as $s\downarrow t$.
\end{proof}

\end{thm}

As a consequence of Theorem \ref{uniqueness}, namely from the comparison property \eqref{comparison}, we obtain the uniqueness part of Theorem \ref{main}. The proof is similar to \cite[Corollary 12]{debus}.

\begin{cor}[Continuity in time]
Let $u:\Omega\times\mt^N\times[0,T]\rightarrow\mr$ be a kinetic solution to \eqref{rovnice}. Then, for all $p\in[1,\infty)$, $u$ has almost surely continuous trajectories in $L^p(\mt^N)$.

\end{cor}

\section{Existence - smooth initial data}

In this section we prove the existence part of Theorem \ref{main} under an additional assumption upon the initial condition: $u_0\in L^p(\Omega;C^\infty(\mt^N)),$ for all $p\in[1,\infty)$.
We employ the vanishing viscosity method, i.e. we approximate the equation \eqref{rovnice} by certain nondegenerate problems, while using also some appropriately chosen approximations $\,\varPhi^\varepsilon,\,B^\varepsilon$ of $\,\varPhi$ and $B$, respectively. These equations have smooth solutions and consequent passage to the limit gives the existence of a kinetic solution to the original equation. Nevertheless, the limit argument is quite technical and has to be done in several steps. It is based on the compactness method: the uniform energy estimates yield tightness of a sequence of approximate solutions and thus, on another probability space, this sequence converges almost surely due to the Skorokhod representation theorem. The limit is then shown to be a martingale kinetic solution to \eqref{rovnice}. Combining this fact and the pathwise uniqueness with the the Gy\"{o}ngy-Krylov characterization of convergence in probability, we finally obtain the desired kinetic solution.

\subsection{Nondegenerate case}
\label{nondegener}

Consider approximations to the identity $(\varphi_\varepsilon), (\psi_\varepsilon)$ on $\mt^N\times\mr$ and $\mr$, respectively, and a truncation $(\chi_\varepsilon)$ on $\mr$, i.e. we define $\chi_\varepsilon(\xi)=\chi(\varepsilon\xi)$, where $\chi$ is a smooth function with bounded support satisfying $0\leq\chi\leq 1$ and
$$\chi(\xi)=\begin{cases}
            1,& \text{if }\;|\xi|\leq\frac{1}{2},\\
	    0,& \text{if }\;|\xi|\geq1.
           \end{cases}
$$
The regularizations of $\varPhi,\,B$ are then defined in the following way
\begin{equation*}
B_i^\varepsilon(\xi)=\big((B_i*\psi_\varepsilon)\chi_\varepsilon\big)(\xi),\;\quad i=1,\dots,N,
\end{equation*}
\begin{equation*}
g^\varepsilon_k(x,\xi)=\begin{cases}
                                \big((g_k*\varphi_\varepsilon)\chi_\varepsilon\big)(x,\xi),& \text{if }\;k\leq\lfloor 1/\varepsilon\rfloor,\\
				0,&\text{if }\;k>\lfloor 1/\varepsilon\rfloor,
                               \end{cases}
\end{equation*}
where $x\in\mt^N,\xi\in\mr$. Consequently, we set $B^\varepsilon=(B_1^\varepsilon,\dots,B_N^\varepsilon)$ and define the operator $\,\varPhi^\varepsilon\,$ by $\,\varPhi^\varepsilon(z)e_k=g_k^\varepsilon(\cdot,z(\cdot)),\,z\in L^2(\mt^N)$.
Clearly, the approximations $B^\varepsilon,\,g_k^\varepsilon$ are of class $C^\infty$ with a compact support therefore Lipschitz continuous. Moreover, the functions $g_k^\varepsilon$ satisfy \eqref{linrust}, \eqref{skorolip} uniformly in $\varepsilon$ and the following Lipschitz condition holds true
\begin{equation}\label{lipp}
\forall x\in\mt^N\quad\forall \xi,\,\zeta\in\mr\quad\sum_{k\geq1}|g_k^\varepsilon(x,\xi)-g_k^\varepsilon(x,\zeta)|^2\leq L_\varepsilon|\xi-\zeta|^2.
\end{equation}
From \eqref{linrust} we conclude that $\varPhi^\varepsilon(z)$ is Hilbert-Schmidt for all $z\in L^2(\mt^N)$.
Also the polynomial growth of $B$ remains valid for $B^\varepsilon$ and holds uniformly in $\varepsilon$.

Suitable approximation of the diffusion matrix $A$ is obtained as its perturbation by $\varepsilon \mathrm{I}$, where $\mathrm{I}$ denotes the identity matrix. We denote $A^\varepsilon=A+\varepsilon \mathrm{I}$.

Consider an approximation of problem \eqref{rovnice} by a nondegenerate equation
\begin{equation}\label{approx}
\begin{split}
\dif u^\varepsilon+ \diver \big(B^\varepsilon(u^\varepsilon)\big)\dif t&=\diver\big(A^\varepsilon(x)\nabla u^\varepsilon\big)\dif t + \varPhi^{\varepsilon}(u^\varepsilon)\,\dif W,\\
u^\varepsilon(0)&=u_0.
\end{split}
\end{equation}

\begin{thm}\label{smooth}
Let $u_0\in L^p(\Omega;C^\infty(\mt^N))$, for all $p\in[1,\infty)$. For any $\varepsilon>0$, there exists a predictable $C^\infty(\mt^N)$-valued process $\,u^\varepsilon$ solving \eqref{approx}.

\begin{proof}
For any fixed $\varepsilon>0$, the assumptions of \cite[Theorem 2.1]{hof} are satisfied and therefore the claim follows.
\end{proof}
\end{thm}

Thus, using the same computation as in the Section \ref{notation}, one can verify that the solution $u^\varepsilon$ satisfies the kinetic formulation of \eqref{approx}: let $f^\varepsilon=\ind_{u^\varepsilon>\xi}$
\begin{equation*}
\begin{split}
\partial_tf^\varepsilon+b^\varepsilon(\xi)\cdotp\nabla f^\varepsilon&-\sum_{i,j=1}^N\partial_{x_i}\big(A_{ij}(x)\partial_{x_j}f^\varepsilon\big)-\varepsilon\Delta f^\varepsilon\\
&\qquad=\delta_{u^\varepsilon=\xi}\varPhi^\varepsilon(u^\varepsilon)\dot{W}+\partial_\xi\bigg(m^\varepsilon-\frac{1}{2}G^2_\varepsilon\delta_{u^\varepsilon=\xi}\bigg),
\end{split}
\end{equation*}
where $m^\varepsilon=n^\varepsilon_1+n^\varepsilon_2$ and both these measures are explicitly known and correspond to the diffusion matrix $A+\varepsilon I$:
$$n_1^\varepsilon=\big(\nabla u^\varepsilon\big)^* A(x)\big(\nabla u^\varepsilon\big)\delta_{u^\varepsilon=\xi},\qquad n^\varepsilon_2=\varepsilon|\nabla u^\varepsilon|^2\delta_{u^\varepsilon=\xi}.$$
Note, that by taking limit in $\varepsilon$ we lose this precise structure of $n_2$. As can be seen from the derivation of kinetic formulation, $u^\varepsilon$ solve \eqref{kinetic} even in a stronger sense than the one from Definition \ref{kinsol}. Indeed, for any $\varphi\in\mc^\infty_c(\mt^N\times\mr),\;t\in[0,T],$ the following holds true $\prst$-a.s.
\begin{align}\label{kinetaprox}
\begin{split}
\big\langle& f^\varepsilon(t),\varphi\big\rangle-\big\langle f_0,\varphi\big\rangle-\int_0^t\big\langle f^\varepsilon(s),b^\varepsilon(\xi)\cdotp\nabla\varphi\big\rangle\,\dif s\\
&\qquad\quad-\int_0^t\Big\langle f^\varepsilon(s),\sum_{i,j=1}^N\partial_{x_j}\big(A_{ij}(x)\partial_{x_i}\varphi\big)\Big\rangle\,\dif s-\varepsilon\int_0^t\big\langle f^\varepsilon(s),\Delta\varphi\big\rangle\,\dif s\\
&=\int_0^t\big\langle\delta_{u^\varepsilon=\xi}\,\varPhi^\varepsilon(u^\varepsilon)\,\dif W,\varphi\big \rangle+\frac{1}{2}\int_0^t\big\langle\delta_{u^\varepsilon=\xi}\,G^2,\partial_\xi\varphi\big\rangle\,\dif s-\big\langle m^\varepsilon,\partial_\xi\varphi\big \rangle([0,t)).
\end{split}
\end{align}

\subsection{Energy estimates}

In this subsection we shall establish the so-called energy estimate that makes it possible to find uniform bounds for approximate solutions and that will later on yield a solution by invoking a compactness argument.

\begin{lemma}
For all $\,\varepsilon\in(0,1)$, for all $\,t\in[0,T]$ and for all $p\in[2,\infty)$, the solution $u^\varepsilon$ satisfies the inequality
\begin{equation}\label{energyest}
\begin{split}
\stred\|u^\varepsilon(t)\|_{L^p(\mt^N)}^p\leq C\big(1+\stred\|u_0\|_{L^p(\mt^N)}^p\big).
\end{split}
\end{equation}

\begin{proof}
We apply the It\^o formula using $f(v)=\|v\|_{L^p(\mt^N)}^p$. If $q$ is the conjugate exponent to $p$ then $f'(v)=p|v|^{p-2}v\in\ml^q(\mt^N)$ and
$$f''(v)=p(p-1)|v|^{p-2}\,\mathrm{Id}\;\;\in\mathscr{L}\big(\ml^p(\mt^N),\ml^q(\mt^N)\big).$$
Therefore
\begin{equation}\label{if1}
\begin{split}
\|u^\varepsilon(t)\|_{L^p(\mt^N)}^p=&\|u_0\|_{L^p(\mt^N)}^p-p\int_0^t\int_{\mt^N}|u^\varepsilon|^{p-2}u^\varepsilon\diver\big(B^\varepsilon(u^\varepsilon)\big)\,\dif x\,\dif s\\
&\;+p\int_0^t\int_{\mt^N}|u^\varepsilon|^{p-2}u^\varepsilon\diver\big(A(x)\nabla u^\varepsilon\big)\,\dif x\,\dif s\\
&\;+\varepsilon p\int_0^t\int_{\mt^N}|u^\varepsilon|^{p-2}u^\varepsilon\Delta u^\varepsilon\,\dif x\,\dif s\\
&\;+p\sum_{k\geq1}\int_0^t\int_{\mt^N}|u^\varepsilon|^{p-2}u^\varepsilon g_k^\varepsilon(x,u^\varepsilon)\,\dif x\,\dif\beta_k(s)\\
&\;+\frac{1}{2}p(p-1)\int_0^t\int_{\mt^N}|u^\varepsilon|^{p-2}G_\varepsilon^2(x,u^\varepsilon)\,\dif x\,\dif s.
\end{split}
\end{equation}
We conclude that the second term on the right hand side vanishes and using the integration by parts, the third one as well as the fourth one is nonpositive. The last term is estimated as follows
\begin{equation*}
\begin{split}
\frac{1}{2}p(p-1)\int_0^t\int_{\mt^N}|u^\varepsilon|^{p-2}G_\varepsilon^2(x,u^\varepsilon)\,\dif x\,\dif s&\leq C\int_0^t\int_{\mt^N}|u^\varepsilon|^{p-2}\big(1+|u^\varepsilon|^2\big)\,\dif x\,\dif s\\
&\leq C\Big(1+\int_0^t\|u^\varepsilon(s)\|_{L^p(\mt^N)}^p\dif s\Big).
\end{split}
\end{equation*}
Finally, expectation and application of the Gronwall lemma yield \eqref{energyest}.
% $$\stred\|u^\varepsilon(t)\|_{L^p(\mt^N)}^p\leq C\big(1+\stred\|u_0\|_{L^p(\mt^N)}^p\big).$$
\end{proof}

\end{lemma}

% Thus, it is easy to see that if $u^\varepsilon_0$ is a mollification of $u_0\in L^\infty(\mt^N)$ the estimates are uniform in $\varepsilon$.

\begin{cor}
The set $\{u^\varepsilon;\,\varepsilon\in(0,1)\}$ is bounded in $\ml^p(\Omega;C([0,T];L^p(\mt^N)))$, for all $p\in[2,\infty)$.

\begin{proof}
To verify the claim, an uniform estimate of $\stred\big(\sup_{0\leq t\leq T}\|u^\varepsilon(t)\|_{L^p(\mt^N)}^p\big)$ is needed. We repeat the approach from the preceding lemma, only for the stochastically forced term we apply the Burkholder-Davis-Gundy inequality.
We have
\begin{equation*}
\begin{split}
\stred\Big(&\sup_{0\leq t\leq T}\|u^\varepsilon(t)\|_{L^p(\mt^N)}^p\Big)\leq\stred\|u_0\|_{L^p(\mt^N)}^p+C\bigg(1+\int_0^T\stred\|u^\varepsilon(s)\|_{L^p(\mt^N)}^p\dif s\bigg)\\
&\qquad\qquad+p\,\stred\bigg(\sup_{0\leq t\leq T}\bigg|\sum_{k\geq1}\int_0^t\int_{\mt^N}|u^\varepsilon|^{p-2}u^\varepsilon g_k^\varepsilon(x,u^\varepsilon)\,\dif x\,\dif\beta_k(s)\bigg|\bigg)
\end{split}
\end{equation*}
and since we are dealing with finite-dimensional Wiener processes, the Burk\-holder-Davis-Gundy, the assumption \eqref{linrust} and the weighted Young inequality yield
\begin{equation*}
\begin{split}
\stred\bigg(\sup_{0\leq t\leq T}&\bigg|\sum_{k=1}^{\lfloor 1/\varepsilon\rfloor}\int_0^t\int_{\mt^N}|u^\varepsilon|^{p-2}u^\varepsilon g_k^\varepsilon(x,u^\varepsilon)\,\dif x\,\dif\beta_k(s)\bigg|\bigg)\\
&\leq C\,\stred\bigg(\int_0^T\bigg(\int_{\mt^N}|u^\varepsilon|^{p-1}\sum_{k= 1}^{\lfloor 1/\varepsilon\rfloor}|g_k^\varepsilon(x,u^\varepsilon)|\,\dif x\bigg)^2\dif s\bigg)^{\frac{1}{2}}\\
&\leq C\,\stred\bigg(\int_0^T\Big(1+\|u^\varepsilon\|_{L^p(\mt^N)}^p\Big)^2\dif s\bigg)^{\frac{1}{2}}\\
&\leq C\bigg(1+\stred\bigg(\int_0^T\|u^\varepsilon(s)\|_{L^p(\mt^N)}^{2p}\,\dif s\bigg)^{\frac{1}{2}}\bigg)\\
% &\leq C\bigg(1+\Big(\stred\big(\sup_{0\leq t\leq T}\|u^\varepsilon(t)\|^p_p\big)\Big)^{\frac{1}{2}}\Big(\stred\int_0^T\|u^\varepsilon(s)\|_p^p\,\dif s\Big)^{\frac{1}{2}}\bigg)\leq\text{\ \tiny $\Big(\sqrt{xy}\leq \frac{x+y}{2}\Big)$}\\
&\leq\frac{1}{2}\,\stred\Big(\sup_{0\leq t\leq T}\|u^\varepsilon(t)\|^p_{L^p(\mt^N)}\Big)+C\bigg(1+\int_0^T\stred\|u^\varepsilon(s)\|_{L^p(\mt^N)}^p\,\dif s\bigg).
\end{split}
\end{equation*}
Therefore
$$\stred\Big(\sup_{0\leq t\leq T}\|u^\varepsilon(t)\|_{L^p(\mt^N)}^p\Big)\leq C\bigg(1+\stred\|u_0\|_{L^p(\mt^N)}^p+\int_0^T\stred\|u^\varepsilon(s)\|_{L^p(\mt^N)}^p\,\dif s\bigg)$$
and the corollary follows from \eqref{energyest}.
\end{proof}

\end{cor}

% \begin{rem}
% We emphasize, that one actually has more. As can be seen from the proof of \cite[Theorem 2.2]{zhang}, it holds true for any $m\in\mn$ and $p\geq 2$ that
% \begin{equation}\label{en}
% \sup_{\varepsilon\in(0,1)}\stred\Big(\sup_{0\leq t\leq T}\|u^\varepsilon(t)\|_{H^m(\mt^N)}^p\Big)<\infty.
% \end{equation}
% Opravdu????? to asi ztratime, budume mit jenom odhad v L^2
% \end{rem}

\subsection{Compactness argument}

To show that there exists $u:\Omega\times\mt^N\times[0,T]\rightarrow\mr$, a kinetic solution to \eqref{rovnice}, one needs to verify the strong convergence of the approximate solutions $u^\varepsilon$. This can be done by combining tightness of their laws with the pathwise uniqueness, which was proved above.

%In this subsection, the following hypothesis upon the initial condition will be needed: $u_0\in L^1(\Omega;W^{\sigma,1}(\mt^N))\cap L^p(\Omega;L^p(\mt^N)),\,p\in[1,\infty),$
%where $\sigma\leq \min\{\frac{\alpha}{\alpha+1},\frac{1}{2}\}$ and $\alpha$ was introduced in \eqref{fceh}.

First, we need to prove a better spatial regularity of the approximate solutions. Towards this end, we introduce two seminorms describing the $W^{\lambda,1}$-regularity of a function $u\in L^1(\mt^N).$ Let $\lambda\in(0,1)$ and define
$$p^{\lambda}(u)=\int_{\mt^N}\int_{\mt^N}\frac{|u(x)-u(y)|}{|x-y|^{N+\lambda}}\,\dif x\,\dif y,$$
$$p^\lambda_\varrho(u)=\sup_{0<\tau<2 D_N}\frac{1}{\tau^\lambda}\int_{\mt^N}\int_{\mt^N}|u(x)-u(y)|\varrho_\tau(x-y)\,\dif x\,\dif y,$$
where $(\varrho_\tau)$ is a fixed regularizing kernel and by $D_N$ we denote the diameter of $[0,1]^N$. The fractional Sobolev space $W^{\lambda,1}(\mt^N)$ is defined as a subspace of $L^1(\mt^N)$ with finite norm
$$\|u\|_{W^{\lambda,1}(\mt^N)}=\|u\|_{L^1(\mt^N)}+p^\lambda(u).$$
According to \cite{debus}, the following relations holds true between these seminorms. Let $s\in(0,\lambda)$, there exists a constant $C=C_{\lambda,\varrho,N}$ such that for all $u\in L^1(\mt^N)$
\begin{equation}\label{seminorm}
p^\lambda_\varrho(u)\leq C p^\lambda(u),\qquad p^s\leq\frac{C}{\lambda-s}\,p^\lambda_\varrho(u).
\end{equation}

\begin{thm}[$W^{\sigma,1}$-regularity]
Set $\sigma=\min\{\frac{\alpha}{\alpha+1},\frac{1}{2}\}$, where $\alpha$ was introduced in \eqref{fceh}. Then there exists a constant $C_T>0$ such that for all $t\in[0,T]$ and all $\varepsilon\in(0,1)$
\begin{equation}\label{spatial}
\stred\, p^\sigma_\varrho\big(u^\varepsilon(t)\big)\leq C_T\big(1+\stred\, p^\sigma_\varrho(u_0)\big).
\end{equation}
In particular, for all $s\in(0,\sigma)$, there exists a constant $C_{T,s,u_0}>0$ such that for all $t\in[0,T]$
\begin{equation}\label{spatial1}
\stred\|u^\varepsilon(t)\|_{W^{s,1}(\mt^N)}\leq C_{T,s,u_0}\big(1+\stred\|u_0\|_{W^{\sigma,1}(\mt^N)}\big).
\end{equation}

\begin{proof}
Proof of this statement is based on Proposition \ref{propdubling}. We have
\begin{equation*}
\begin{split}
&\stred\int_{(\mt^N)^2}\int_\mr\varrho_\tau(x-y)f^\varepsilon(x,t,\xi)\bar{f}^\varepsilon(y,t,\xi)\,\dif\xi\,\dif x\,\dif y\\
&\leq\stred\int_{(\mt^N)^2}\int_{\mr^2}\varrho_\tau(x-y)\psi_\delta(\xi-\zeta)f^\varepsilon(x,t,\xi)\bar{f}^\varepsilon(y,t,\zeta)\,\dif\xi\,\dif\zeta\,\dif x\,\dif y+\delta\\
&\leq\stred\!\int_{(\mt^N)^2}\!\int_{\mr^2}\!\varrho_\tau(x-y)\psi_\delta(\xi-\zeta)f_{0}(x,\xi)\bar{f}_{0}(y,\zeta)\dif\xi\dif\zeta\dif x\dif y+\delta+I+J+K\!\\
&\leq\stred\int_{(\mt^N)^2}\int_\mr\varrho_\tau(x-y)f_{0}(x,\xi)\bar{f}_{0}(y,\xi)\,\dif\xi\,\dif x\,\dif y+2\delta+I+J+K.
\end{split}
\end{equation*}
%Moreover, the convolution error term $\eta_t(\delta)$ can be computed explicitly. Indeed,
%\begin{equation*}
%\begin{split}
%&\eta_t(\delta)=\stred\int_{(\mt^N)^2}\int_\mr\varrho_\tau(x-y)f^\varepsilon(x,t,\xi)\bar{f}^\varepsilon(y,t,\xi)\,\dif\xi\,\dif x\,\dif y\\
%&\qquad\qquad-\stred\int_{(\mt^N)^2}\int_{\mr^2}\varrho_\tau(x-y)\psi_\delta(\xi-\zeta)f^\varepsilon(x,t,\xi)\bar{f}^\varepsilon(y,t,\zeta)\,\dif\xi\,\dif\zeta\,\dif x\,\dif y\\
%&=\stred\!\int_{(\mt^N)^2}\!\int_\mr\!\varrho_\tau(x-y)f^\varepsilon(x,t,\xi)\bigg[\bar{f}^\varepsilon(y,t,\xi)-\int_\mr\psi_\delta(\xi-\zeta)\bar{f}^\varepsilon(y,t,\zeta)\,\dif\zeta\bigg]\dif\xi\dif x\dif y\\
%&=\stred\!\int_{(\mt^N)^2}\!\int_\mr\!\varrho_\tau(x-y)\,\ind_{u^\varepsilon(x,t)>\xi}\int_{\mr}\psi_\delta(\xi-\zeta)\Big[\ind_{u^\varepsilon(y,t)\leq\xi}-\ind_{u^\varepsilon(y,t)\leq\zeta}\Big]\dif\zeta\dif\xi\dif x\dif y\\
%&\leq\stred\int_{(\mt^N)^2}\int_\mr\varrho_\tau(x-y)\,\ind_{u^\varepsilon(x,t)>\xi}\int_{\xi-\delta}^\xi\psi_\delta(\xi-\zeta)\,\ind_{\zeta<u^\varepsilon(y,t)\leq\xi}\,\dif\zeta\,\dif\xi\,\dif x\,\dif y\\
%% &\leq\stred\int_{(\mt^N)^2}\varrho_\tau(x-y)\int_{u^\varepsilon(y,t)}^{u^\varepsilon(y,t)+\delta}\dif \xi\,\dif x\,\dif y=\delta.\\
%&\leq\stred\int_{(\mt^N)^2}\varrho_\tau(x-y)\int_{u^\varepsilon(y,t)}^{\min\{u^\varepsilon(x,t),u^\varepsilon(y,t)+\delta\}}\dif \xi\,\dif x\,\dif y\leq\delta.
%\end{split}
%\end{equation*}
From the same estimates as the ones used in the proof of Theorem \ref{uniqueness}, we conclude
\begin{equation*}
\begin{split}
&\stred\int_{(\mt^N)^2}\varrho_\tau(x-y)\big(u^\varepsilon(x,t)-u^\varepsilon(y,t)\big)^+\,\dif x\,\dif y\\
&\quad\leq\stred\int_{(\mt^N)^2}\varrho_\tau(x-y)\big(u_0(x)-u_0(y)\big)^+\,\dif x\,\dif y+2\delta+Ct\big(\delta^{-1}\tau+\delta^{-1}\tau^2+\delta^\alpha\big)
\end{split}
\end{equation*}
By optimization in $\delta$, i.e. setting $\delta=\tau^\beta$, we obtain
$$\sup_{0<\tau<2D_N}\frac{2\delta+Ct\big(\delta^{-1}\tau+\delta^{-1}\tau^2+\delta^\alpha\big)}{\tau^\sigma}\leq C_t,$$
where the maximal choice of the parameter $\sigma$ is $\min\big\{\frac{\alpha}{\alpha+1},\frac{1}{2}\big\}$ which corresponds to $\beta=\max\big\{\frac{1}{\alpha+1},\frac{1}{2}\big\}.$
Hence we deduce \eqref{spatial}. Furthermore, by \eqref{seminorm} it holds
$$\stred\,p^s(u^\varepsilon(t))\leq C_{T,s}(1+\stred\, p^\sigma(u_0))$$
and from \eqref{energyest} we obtain
$$\stred\|u^\varepsilon(t)\|_{L^1(\mt^N)}\leq\stred\|u^\varepsilon(t)\|_{L^2(\mt^N)}\leq C_T\Big(1+\big(\,\stred\|u_0\|^2_{L^2(\mt^N)}\big)^\frac{1}{2}\Big)$$
and \eqref{spatial1} follows. As a consequence of the previous estimate, the constant in \eqref{spatial1} depends on the $L^2(\Omega;L^2(\mt^N))$-norm of the initial condition.
% and the fractional Poincaré inequality
% $$\bigg\|u^\varepsilon(t)-\int_{\mt^N}\!u^\varepsilon(t)\,\dif x\bigg\|_{L^1(\mt^N)}\leq C_s\,p^s\big(u^\varepsilon(t)\big)$$
% gives
% $$\stred\|u^\varepsilon(t)\|_{L^1(\mt^N)}\leq C_s\,\stred\,p^s\big(u^\varepsilon(t)\big)+\stred\bigg|\int_{\mt^N}\!u^\varepsilon(t)\dif x\bigg|.$$
% We have
% \begin{equation*}
% \stred\bigg|\int_{\mt^N}\!u^\varepsilon(t)\dif x\bigg|\leq\stred\bigg|\int_{\mt^N}\!u_0\dif x\bigg|+\stred\bigg\|\int_0^t\varPhi^\varepsilon(u^\varepsilon)\dif W\bigg\|_{L^2(\mt^N)},
% \end{equation*}
% where the stochastic term can be estimated, using the Burkholder-Davis-Gundy inequality and the energy estimate, by some constant dependent only on $T$ and $u_0$. More precisely, the dependance on the initial condition is given by the energy estimate so the constant actually depends on its $L^2(\Omega;L^2(\mt^N))$-norm. We conclude
% $$\stred\|u^\varepsilon(t)\|_{L^1(\mt^N)}\leq C_{T,s,u_0}\big(1+\stred\,p^\sigma(u_0)+\stred\|u_0\|_{L^1(\mt^N)}\big)$$
% which yields \eqref{spatial1}.
\end{proof}

\end{thm}

\begin{cor}\label{inter}
For all $\gamma\in(0,\sigma)$ and $q>1$ satisfying $\gamma q<\sigma$, there exists a constant $C>0$ such that for all $\varepsilon\in(0,1)$
\begin{equation}\label{odhad1}
\stred\|u^\varepsilon\|^q_{L^q(0,T;W^{\gamma,q}(\mt^N))}\leq C.
\end{equation}

\begin{proof}
The claim is a consequence of the bounds \eqref{energyest} and \eqref{spatial1}. Indeed, fix $s\in(0,\sigma)$ and $p\in(1,\infty)$. We will use the interpolation inequality
\begin{equation}\label{interpolace}
\|\cdot\|_{W^{\gamma,q}(\mt^N)}\leq C\|\cdot\|^{1-\theta}_{W^{\gamma_0,q_0}(\mt^N)}\|\cdot\|^\theta_{W^{\gamma_1,q_1}(\mt^N)},
\end{equation}
where $\gamma_0,\gamma_1\in\mr,\,q_0,q_1\in(0,\infty),$ $\gamma=(1-\theta)\gamma_0+\theta \gamma_1,$ $\frac{1}{q}=\frac{1-\theta}{q_0}+\frac{\theta}{q_1},$ $\theta\in(0,1).$ In particular, fix $s\in(\gamma q,\sigma)$ and set $\gamma_0=s,$ $\gamma_1=0,$ $q_0=1,$ $q_1=p$. Then we obtain $\theta=\frac{s-\gamma}{s},\,p=\frac{(s-\gamma)q}{s-\gamma q}$ and
\begin{equation*}
\begin{split}
\stred\|u^\varepsilon(t)\|^q_{W^{\gamma,q}(\mt^N)}&\leq C\,\stred\Big(\|u^\varepsilon(t)\|^{(1-\theta)q}_{W^{s,1}(\mt^N)}\|u^\varepsilon(t)\|^{\theta q}_{L^p(\mt^N)}\Big)\\
&\leq C\Big(\stred\|u^\varepsilon(t)\|_{W^{s,1}(\mt^N)}\Big)^{(1-\theta)q}\Big(\stred\|u^\varepsilon(t)\|_{L^p(\mt^N)}^p\Big)^{1-(1-\theta)q}\leq C.
\end{split}
\end{equation*}
%where the conditions $2(1-\theta)q=1$ and $2\theta q=p$ yield $\gamma=\frac{s}{1+p},\,q=\frac{1+p}{2}.$
\end{proof}

\end{cor}

Also a better time regularity is needed.

\begin{lemma}\label{jjkj}
Suppose that $\lambda\in(0,1/2),\,q\in[2,\infty).$ There exists a constant $C>0$ such that for all $\varepsilon\in(0,1)$
\begin{equation}\label{odhad3}
\stred\|u^\varepsilon\|^q_{C^\lambda([0,T];H^{-2}(\mt^N))}\leq C.
\end{equation}

\begin{proof}
Let $q\in[2,\infty)$. Recall that the set $\{u^\varepsilon;\,\varepsilon\in(0,1)\}$ is bounded in
$$L^q(\Omega;C(0,T;L^q(\mt^N))).$$
Since all $B^\varepsilon$ have the same polynomial growth we conclude, in particular, that
$$\{\diver(B^\varepsilon(u^\varepsilon))\},\quad\{\diver(A(x)\nabla u^\varepsilon)\},\quad\{\varepsilon\Delta u^\varepsilon\}$$
are bounded in $L^q(\Omega;C(0,T;H^{-2}(\mt^N)))$ and consequently
$$\stred\Big\|u^\varepsilon-\int_0^\tec\varPhi^\varepsilon(u^\varepsilon)\,\dif W\Big\|^q_{C^1([0,T];H^{-2}(\mt^N))}\leq C.$$

In order to deal with the stochastic integral, let us recall the definition of the Riemann-Liouville operator: let $X$ be a Banach space, $p\in(1,\infty],\,\alpha\in(1/p,1]$ and $f\in L^p(0,T;X)$, then we define
$$\big(R_\alpha f\big)(t)=\frac{1}{\Gamma(\alpha)}\int_0^t(t-s)^{\alpha-1}f(s)\,\dif s,\qquad t\in[0,T].$$
It is well known that $R_\alpha$ is a bounded linear operator from $L^p(0,T;X)$ to the space of H\"{o}lder continuous functions $C^{\alpha-1/p}([0,T];X)$ (see e.g. \cite[Theorem 3.6]{samko}).
Assume now that $q\in(2,\infty),$ $\alpha\in(1/q,1/2)$. Then according to the stochastic Fubini theorem \cite[Theorem 4.18]{prato}
\begin{equation*}
\begin{split}
\int_0^t\varPhi^\varepsilon\big(u^\varepsilon(s)\big)\,\dif W(s)=\big(R_\alpha Z\big)(t),
\end{split}
\end{equation*}
where
\begin{equation*}
\begin{split}
Z(s)=\frac{1}{\Gamma(1-\alpha)}\int_0^s(s-r)^{-\alpha}\,\varPhi^\varepsilon\big(u^\varepsilon(r)\big)\,\dif W(r).
\end{split}
\end{equation*}
Therefore using the Burkholder-Davis-Gundy and Young inequality and the estimate \eqref{linrust}
\begin{equation*}
\begin{split}
\stred\bigg\|\int_0^\tec\varPhi^\varepsilon(u^\varepsilon)&\,\dif W\bigg\|_{C^{\alpha-1/q}([0,T];L^2(\mt^N))}^q\leq C\,\stred\|Z\|_{L^q(0,T;L^2(\mt^N))}^q\\
%&=C\,\stred\int_0^T\bigg\|\int_0^t(t-s)^{-\alpha}\varPhi^\varepsilon\big(u^\varepsilon(s)\big)\,\dif W(s)\bigg\|_{L^2(\mt^N)}^q\dif t\\
&\leq C\,\int_0^T\stred\bigg(\int_0^t\frac{1}{(t-s)^{2\alpha}}\|\varPhi^\varepsilon(u^\varepsilon)\|_{L_2(\mathfrak{U};L^2(\mt^N))}^2\dif s\bigg)^\frac{q}{2}\dif t\\
&\leq CT^{\frac{q}{2}(1-2\alpha)}\stred\int_0^T\Big(1+\|u^\varepsilon(s)\|_{L^2(\mt^N)}^q\Big)\dif s\\
&\leq CT^{\frac{q}{2}(1-2\alpha)}\Big(1+\|u^\varepsilon\|_{L^q(\Omega;L^q(0,T;L^2(\mt^N)))}^q\Big)\leq C
\end{split}
\end{equation*}
and the claim follows.
\end{proof}
\end{lemma}

\begin{cor}\label{jjk}
For all $\vartheta>0$ there exist $\beta>0$ and $C>0$ such that for all $\varepsilon\in(0,1)$
\begin{equation}\label{odhad4}
\stred\|u^\varepsilon\|_{C^\beta([0,T];H^{-\vartheta}(\mt^N))}\leq C.
\end{equation}

\begin{proof}
The proof follows easily from interpolation between \eqref{odhad3} and \eqref{energyest}.
\end{proof}
\end{cor}

\begin{cor}
Suppose that $\kappa\in(0,\frac{\sigma}{2(4+\sigma)})$. There exists a constant $C>0$ such that for all $\varepsilon\in(0,1)$
\begin{equation}\label{odhad2}
\stred\|u^\varepsilon\|_{H^{\kappa}(0,T;L^2(\mt^N))}\leq C.
\end{equation}

\begin{proof}
%Fix $\vartheta=\gamma\in(0,\sigma)$ (cf. Corollary \ref{inter}).
It follows from Lemma \ref{jjkj} that
\begin{equation}\label{odhad5}
\stred\|u^\varepsilon\|^q_{H^{\lambda}(0,T;H^{-2}(\mt^N))}\leq C,
\end{equation}
where $\lambda\in(0,1/2),\,q\in[1,\infty)$. Let $\gamma\in(0,\sigma/2)$. If $\kappa=\theta\lambda\,$ and $\,0=-2\theta+(1-\theta)\gamma$ then it follows by the interpolation and the H\"{o}lder inequality
\begin{equation*}
\begin{split}
\stred\|u^\varepsilon\|_{H^{\kappa}(0,T;L^2(\mt^N))}&\leq C\,\stred\Big(\|u^\varepsilon\|_{H^\lambda(0,T;H^{-2}(\mt^N))}^\theta\|u^\varepsilon\|_{L^2(0,T;H^\gamma(\mt^N))}^{1-\theta}\Big)\\
&\leq C\Big(\stred\|u^\varepsilon\|_{H^\lambda(0,T;H^{-2}(\mt^N))}^{\theta p}\Big)^{\frac{1}{p}}\Big(\stred\|u^\varepsilon\|_{L^2(0,T;H^\gamma(\mt^N))}^{(1-\theta)r}\Big)^\frac{1}{r},
\end{split}
\end{equation*}
where the exponent $r$ is chosen in order to satisfy $(1-\theta)r=2$. The proof now follows from \eqref{odhad1} and \eqref{odhad5}.
\end{proof}
\end{cor}

Now, we have all in hand to show tightness of the collection $\{\mu^\varepsilon;\,\varepsilon\in(0,1)\}$ in $\mathcal{X}.$ Let us define the path space $\mathcal{X}=\mathcal{X}_u\times\mathcal{X}_W$, where
\begin{equation}\label{pathspace}
\mathcal{X}_u=L^1\big(0,T;L^1(\mt^N)\big)\cap C\big([0,T];H^{-1}(\mt^N)\big),\qquad\mathcal{X}_W=C\big([0,T];\mathfrak{U}_0\big).
\end{equation}
For all $\varepsilon\in(0,1)$ we denote by $\mu_{u^\varepsilon}$ the law of $u^\varepsilon$ on $\mathcal{X}_u$ and by $\mu_W$ the law of $W$ on $\mathcal{X}_W$. Their joint law on $\mathcal{X}$ is then denoted by $\mu^\varepsilon=\mu_{u^\varepsilon}\otimes\mu_W$.

\begin{thm}\label{tight}
The set $\{\mu^\varepsilon;\,\varepsilon\in(0,1)\}$ is tight and therefore relatively weakly compact in $\mathcal{X}.$
\begin{proof}
% First, we employ an Aubin-Dubinskii type compact embedding theorem which, in our setting, reads (see for instance \cite{amann}):
% Suppose that
% \begin{enumerate}
%  \item $\mathcal{V}$ is a bounded subset of $L^1(0,T;W^{s,1}(\mt^N)),$
%  \item $\|v(\cdot+h)-v\|_{L^1(0,T-h;L^1(\mt^N))}\leq C h^\beta,\quad\beta\in(0,1),\;h\in(0,T),\;v\in\mathcal{V}.$
%  %\item $\lim_{h\rightarrow0_+}\|v(\cdot+h)-v\|_{L^1(0,T-h;L^1(\mt^N))}=0,\quad\text{uniformly for }v\in\mathcal{V}.$
% \end{enumerate}
% Then $\mathcal{V}$ is relatively compact in $L^1(0,T;L^1(\mt^N))$.
%
% For $R>0$ and $\lambda\in(0,1/2)$ we define the sets
% $$B_{1,R}=\big\{u\in L^1(0,T;W^{s,1}(\mt^N));\,\|u\|_{L^1(0,T;W^{s,1}(\mt^N))}\leq R\big\},$$
% $$B_{2,R}=\big\{u\in L^1(0,T;L^1(\mt^N));\,\|u(\cdot+h)-u\|_{L^1(0,T-h;L^1(\mt^N))}\leq R h^{\lambda},\;h\in(0,T)\big\}.$$
% Since $B_{1,R}\cap B_{2,R}$ is compact in $\mathcal{X}_u$ and
% \begin{equation*}
% \begin{split}
% \mu_{u^\varepsilon}&\big([B_{1,R}\cap B_{2,R}]^C\big)\leq \mu_{u^\varepsilon}\big(B_{1,R}^C\big)+\mu_{u^\varepsilon}\big(B_{2,R}^C\big)\\
% &=\prst\Big(\|u^\varepsilon\|_{L^1(0,T;W^{s,1}(\mt^N))}>R\Big)+\prst \Big(\|u^\varepsilon(\cdot+h)-u^\varepsilon\|_{L^1(0,T-h;L^1(\mt^N))}>R h^\lambda\Big)
% \end{split}
% \end{equation*}

First, we employ an Aubin-Dubinskii type compact embedding theorem which, in our setting, reads (see \cite{lions1} for a general exposition; the proof of the following version can be found in \cite{fland}):
$$L^2(0,T;H^{\gamma}(\mt^N))\cap H^{\kappa}(0,T;L^2(\mt^N))\overset{c}{\hookrightarrow}L^2(0,T;L^2(\mt^N))\hookrightarrow L^1(0,T;L^1(\mt^N)).$$
For $R>0$ we define the set
\begin{multline*}
B_{1,R}=\{u\in L^2(0,T;H^{\gamma}(\mt^N))\cap H^{\kappa}(0,T;L^2(\mt^N));\\
\|u\|_{L^2(0,T;H^{\gamma}(\mt^N))}+\|u\|_{H^{\kappa}(0,T;L^2(\mt^N))}\leq R\}
\end{multline*}
which is thus compact in $L^1(0,T;L^1(\mt^N))$. Moreover, by \eqref{odhad1} and \eqref{odhad2}
\begin{equation*}
\begin{split}
\mu_{u^\varepsilon}\big(B_{1,R}^C\big)&\leq\prst\bigg(\|u^\varepsilon\|_{L^2(0,T;H^{\gamma}(\mt^N))}>\frac{R}{2}\bigg)+\prst\bigg(\|u^\varepsilon\|_{H^{\kappa}(0,T;L^2(\mt^N))}>\frac{R}{2}\bigg)\\
&\leq\frac{2}{R}\Big(\stred\|u^\varepsilon\|_{L^2(0,T;H^{\gamma}(\mt^N))}+\stred\|u^\varepsilon\|_{H^{\kappa}(0,T;L^2(\mt^N))}\Big)\leq\frac{C}{R}.
\end{split}
\end{equation*}
In order to prove tightness in $C([0,T];H^{-1}(\mt^N))$ we employ the compact embedding
$$C^\beta([0,T];H^{-\vartheta}(\mt^N))\overset{c}{\hookrightarrow}C^{\tilde{\beta}}([0,T];H^{-1}(\mt^N))\hookrightarrow C([0,T];H^{-1}(\mt^N)),$$
where $\tilde{\beta}<\beta,\,0<\vartheta<1.$
Define
$$B_{2,R}=\{u\in C^\beta([0,T];H^{-\vartheta}(\mt^N));\,\|u\|_{C^\beta([0,T];H^{-\vartheta}(\mt^N))}\leq R\}$$
then by \eqref{odhad4}
\begin{equation*}
\begin{split}
\mu_{u^\varepsilon}\big(B_{2,R}^C\big)\leq\frac{1}{R}\stred\|u^\varepsilon\|_{C^\beta([0,T];H^{-\vartheta}(\mt^N))}\leq\frac{C}{R}.
\end{split}
\end{equation*}
Let $\eta>0$ be given. Then, since $B_R=B_{1,R}\cap B_{2,R}$ is compact in $\mathcal{X}_u$ and for some suitably chosen $R>0$ it holds true
$$\mu_{u^\varepsilon}(B_R)\geq 1-\frac{\eta}{2},$$
we obtain the tightness of $\{\mu_{u^\varepsilon};\,\varepsilon\in(0,1)\}$.
Since also the law $\mu_W$ is tight as being a Radon measure on the Polish space $\mathcal{X}_W$, there exists a compact set $C_\eta\subset\mathcal{X}_W$ such that
$\mu_W(C_\eta)\geq 1-\frac{\eta}{2}.$
We conclude that $B_R\times C_\eta$ is compact in $\mathcal{X}$ and $\mu^\varepsilon(B_R\times C_\eta)\geq 1-\eta$. Thus, $\{\mu^\varepsilon;\,\varepsilon\in(0,1)\}$ is tight in
$\mathcal{X}$
and Prokhorov's theorem therefore implies that it is also relatively weakly compact.
\end{proof}
\end{thm}

Passing to a weakly convergent subsequence $\mu^n=\mu^{\varepsilon_n}$ (and denoting by $\mu$ the limit law) we now apply the Skorokhod representation theorem to infer the following proposition.

\begin{prop}\label{skorokhod}
There exists a probability space $(\tilde{\Omega},\tilde{\mf},\tilde{\prst})$ with a sequence of $\mathcal{X}$-valued random variables $(\tilde{u}^n,\tilde{W}^n),\,n\in\mn,$ and $(\tilde{u},\tilde{W})$ such that
\begin{enumerate}
 \item the laws of $(\tilde{u}^n,\tilde{W}^n)$ and $(\tilde{u},\tilde{W})$ under $\,\tilde{\prst}$ coincide with $\mu^n$ and $\mu$, respectively,
 \item $(\tilde{u}^n,\tilde{W}^n)$ converges $\,\tilde{\prst}$-almost surely to $(\tilde{u},\tilde{W})$ in the topology of $\mathcal{X}$,
\end{enumerate}

\end{prop}

Note, that we can assume, without loss of generality, that the $\sigma$-algebra $\tilde{\mf}$ is countably generated. This fact will be used later on for the application of the Banach-Alaoglu theorem.

\begin{rem}
It should be noted that the energy estimates remain valid also for the candidate solution $\tilde{u}$. Indeed, let $p\in[2,\infty)$
\begin{equation*}
\begin{split}
\tilde{\stred}\sup_{0\leq t\leq T}\|\tilde{u}(t)\|^p_{L^p(\mt^N)}&\leq\liminf_{n\rightarrow\infty}\,\tilde{\stred}\sup_{0\leq t\leq T}\|\tilde{u}^n(t)\|^p_{L^p(\mt^N)}\\
&=\liminf_{n\rightarrow\infty}\,\stred\sup_{0\leq t\leq T}\|u^n(t)\|^p_{L^p(\mt^N)}\leq C.
\end{split}
\end{equation*}
\end{rem}

Finally, let $(\tilde{\mf}_t)$ be the $\tilde{\prst}$-augmented canonical filtration of the process $(\tilde{u},\tilde{W})$, that is
$$\tilde{\mf}_t=\sigma\big(\sigma\big(\varrho_t\tilde{u},\varrho_t\tilde{W}\big)\cup\big\{N\in\tilde{\mf};\;\tilde{\prst}(N)=0\big\}\big),\quad t\in[0,T],$$
where $\varrho_t$ denotes the operator of restriction to the interval $[0,t]$.

\subsection{Passage to the limit}\label{passage}

In this paragraph we provide the technical details of the identification of the limit process with a kinetic solution. The technique performed here will be used also in the proof of existence of a pathwise kinetic solution.
\begin{thm}\label{pass}
The triple
$\big((\tilde{\Omega},\tilde{\mf},(\tilde{\mf}_t),\tilde{\prst}),\tilde{W},\tilde{u}\big)$
is a martingale kinetic solution to the problem \eqref{rovnice}.
\end{thm}
% We state this result in the next Proposition, because it will be used also in the proof of existence of a pathwise solution.

Let us define
$$f^n=\ind_{u^n>\xi},\quad\tilde{f}^n=\ind_{\tilde{u}^n>\xi},\quad\tilde{f}=\ind_{\tilde{u}>\xi},$$
$$m^n=n_1^n+n_2^n=\big(\nabla u^n\big)^*A(x)\big(\nabla u^n\big)\delta_{u^n=\xi}+\varepsilon_n\big|\nabla u^n\big|^2\delta_{u^n=\xi},$$
$$\tilde{m}^n=\tilde{n}_1^n+\tilde{n}_2^n=\big(\nabla \tilde{u}^n\big)^*A(x)\big(\nabla \tilde{u}^n\big)\delta_{\tilde{u}^n=\xi}+\varepsilon_n\big|\nabla \tilde{u}^n\big|^2\delta_{\tilde{u}^n=\xi}.$$
Let $\mathcal{M}_b$ denote the space of bounded Borel measures over $\mt^N\times[0,T]\times\mr$, i.e. the dual space of $C_b$, the set of continuous bounded functions on $\mt^N\times[0,T]\times\mr$.

\begin{lemma}\label{weakstar}
It holds true (up to subsequences)
\begin{enumerate}
 \item $\tilde{f}^n\overset{w^*}{\longrightarrow}\tilde{f}\;\quad\text{in}\quad L^\infty(\tilde{\Omega}\times\mt^N\times[0,T]\times\mr)\text{-weak}^*,$
 \item there exists a kinetic measure $\tilde{m}$ such that
\begin{equation}\label{konverg}
\tilde{m}^n\overset{w^*}{\longrightarrow}\tilde{m}\;\quad\text{in}\quad L^2(\tilde{\Omega};\mathcal{M}_b)\text{-weak}^*.
\end{equation}
Moreover, $\tilde{m}$ can be rewritten as $\,\tilde{n}_1+\tilde{n}_2$, where $\,\tilde{n}_1=(\nabla \tilde{u})^*A(x)(\nabla \tilde{u})\delta_{\tilde{u}=\xi}$ and $\tilde{n}_2$ is almost surely a nonnegative measure over $\,\mt^N\times[0,T]\times\mr$.
\end{enumerate}

\begin{proof}
According to Proposition \ref{skorokhod}, there exists a set $\Sigma\subset\tilde{\Omega}\times\mt^N\times[0,T]$ of full measure and a subsequence still denoted by $\{\tilde{u}^n;\,n\in\mn\}$ such that $\tilde{u}^n(\omega,x,t)\rightarrow\tilde{u}(\omega,x,t)$ for all $(\omega,x,t)\in\Sigma$. We infer that
\begin{equation}\label{c}
\ind_{\tilde{u}^n(\omega,x,t)>\xi}\longrightarrow\ind_{\tilde{u}(\omega,x,t)>\xi}
\end{equation}
whenever
$$\Big(\tilde{\prst}\otimes\mathcal{L}_{\,\mt^N}\otimes\mathcal{L}_{[0,T]}\Big)\big\{(\omega,x,t)\in\Sigma;\,\tilde{u}(\omega,x,t)=\xi\big\}=0,$$
where by $\mathcal{L}_{\,\mt^N},\,\mathcal{L}_{[0,T]}$ we denoted the Lebesque measure on $\mt^N$ and $[0,T]$, respectively.
However, the set
$$D=\Big\{\xi\in\mr;\,\Big(\tilde{\prst}\otimes\mathcal{L}_{\,\mt^N}\otimes\mathcal{L}_{[0,T]}\Big)\big(\tilde{u}=\xi\big)>0\Big\}$$
is at most countable since we deal with finite measures. To obtain a contradiction, suppose that $D$ is uncountable and denote
$$D_k=\Big\{\xi\in\mr;\,\Big(\tilde{\prst}\otimes\mathcal{L}_{\,\mt^N}\otimes\mathcal{L}_{[0,T]}\Big)\big(\tilde{u}=\xi\big)>\frac{1}{k}\Big\},\quad k\in\mn.$$
Then $D=\cup_{k\in\mn}D_k$ is a countable union so there exists $k_0\in\mn$ such that $D_{k_0}$ is uncountable. Hence
\begin{equation*}
\begin{split}
\Big(\tilde{\prst}\otimes\mathcal{L}_{\,\mt^N}\otimes\mathcal{L}_{[0,T]}\Big)&\big(\tilde{u}\in D\big)\geq\Big(\tilde{\prst}\otimes\mathcal{L}_{\,\mt^N}\otimes\mathcal{L}_{[0,T]}\Big)\big(\tilde{u}\in D_{k_0}\big)\\
&=\sum_{\xi\in D_{k_0}}\Big(\tilde{\prst}\otimes\mathcal{L}_{\,\mt^N}\otimes\mathcal{L}_{[0,T]}\Big)\big(\tilde{u}=\xi\big)>\sum_{\xi\in D_{k_0}}\frac{1}{k_0}=\infty
\end{split}
\end{equation*}
and the desired contradiction follows.
We conclude that the convergence in \eqref{c} holds true for a.e. $(\omega,x,t,\xi)$ and obtain (i) by the dominated convergence theorem.

As the next step we shall show that the set $\{\tilde{m}^n;\;n\in\mn\}$ is bounded in $\ml^2(\tilde{\Omega};\mathcal{M}_b)$. Indeed, with regard to the computations used in proof of the energy inequality, we get from \eqref{if1}
\begin{equation*}
\begin{split}
\int_0^T\!\int_{\mt^N}\!&\big(\nabla u^n\big)^*\!A(x)\big(\nabla u^n\big)\dif x\,\dif t+\varepsilon_n\int_0^T\!\int_{\mt^N}\!\big|\nabla u^n\big|^2\dif x\,\dif t\leq C\|u_0\|_{L^2(\mt^N)}^2\\
&+C\sum_{k\geq 1}\int_0^T\!\int_{\mt^N}u^ng_k^n(x,u^n)\dif x\,\dif \beta_k(t)+C\int_0^T\!\int_{\mt^N}G_k^2(x,u^n)\dif x\,\dif s.
\end{split}
\end{equation*}
Taking square and expectation and finally by the It\^o isometry, we deduce
\begin{equation*}
\begin{split}
\tilde{\stred}\big|\tilde{m}^n&(\mt^N\times[0,T]\times\mr)\big|^2=\stred\big|m^n(\mt^N\times[0,T]\times\mr)\big|^2\\
\,&=\stred\bigg|\int_0^T\!\int_{\mt^N}\!\big(\nabla u^n\big)^*\!A(x)\big(\nabla u^n\big)\dif x\,\dif t+\varepsilon_n\!\int_0^T\!\int_{\mt^N}\!\big|\nabla u^n\big|^2\dif x\,\dif t\bigg|^2\!\leq C.
\end{split}
\end{equation*}
Thus, according to the Banach-Alaoglu theorem, \eqref{konverg} is obtained (up to subsequence). However, it still remains to show that the weak* limit $\tilde{m}$ is actually a kinetic measure.
First, we conclude from \eqref{energyest}
$$\tilde{\stred}\int_{\mt^N\times[0,T]\times\mr}|\xi|^{p-2}\dif \tilde{m}^n(x,t,\xi)\leq C$$
so for all $k>0$
\begin{equation*}
\begin{split}
\tilde{\stred}\int_{\mt^N\times[0,T]\times\mr}\min&\big(|\xi|^{p-2},k\big)\dif \tilde{m}(x,t,\xi)\\
&=\lim_{n\rightarrow\infty}\int_{\mt^N\times[0,T]\times\mr}\min\big(|\xi|^{p-2},k\big)\dif \tilde{m}^n(x,t,\xi)\leq C
\end{split}
\end{equation*}
and consequently
$$\tilde{\stred}\int_{\mt^N\times[0,T]\times\mr}\dif \tilde{m}^n(x,t,\xi)\leq C.$$
Considering this fact and taking $\phi_R=\ind_{|\xi|\geq R}$ we have by the dominated convergence theorem
$$\lim_{R\rightarrow\infty}\tilde{\stred} \tilde{m} \big(\mt^N\times[0,T]\times B_R^c\big)=\tilde{\stred}\int_{\mt^N\times[0,T]\times\mr}\lim_{R\rightarrow\infty}\phi_R(x,t,\xi)\dif \tilde{m}(x,t,\xi)=0,$$
where $B_R^c=\{\xi\in\mr;\;|\xi|\geq R\}$. Hence $\tilde{m}$ vanishes for large $\xi$.
For predictability of
$$t\longmapsto\int_{\mt^N\times[0,t]\times\mr}\psi(x,\xi)\dif \tilde{m}(x,s,\xi),\qquad\psi\in\mc_b(\mt^N\times\mr),$$
the same arguments as in \cite{debus} can be used.

Finally, by the same approach as above, we deduce that there exist kinetic measures $\tilde{o}_1,\,\tilde{o}_2\,$ such that
$$\tilde{n}_1^n\overset{w^*}{\longrightarrow}\tilde{o}_1,\quad\tilde{n}_2^n\overset{w^*}{\longrightarrow}\tilde{o}_2\qquad\text{in}\quad L^2(\tilde{\Omega};\mathcal{M}_b)\text{-weak}^*.$$
Recall, that by $\sigma(x)$ we denoted the square-root matrix of $A(x)$. Then from \eqref{if1} we obtain
$$\tilde{\stred}\int_0^T\int_{\mt^N}\big(\nabla \tilde{u}^n\big)^*A(x)\big(\nabla \tilde{u}^n\big)\,\dif x\,\dif t=\tilde{\stred}\int_0^T\int_{\mt^N}\big|\sigma(x)(\nabla \tilde{u}^n)\big|^2\dif x\,\dif t\leq C$$
hence application of the Banach-Alaoglu theorem yields that, up to subsequence, $\sigma(x)(\nabla \tilde{u}^n)$ converges weakly in $L^2(\tilde{\Omega}\times\mt^N\times[0,T])$. On the other hand, from the strong convergence given by Proposition \ref{skorokhod}, we conclude using integration by parts, for all $\psi\in C^1(\mt^N\times[0,T])$,
$$\int_0^T\int_{\mt^N}\sigma(x)(\nabla \tilde{u}^n)\psi(x,t)\,\dif x\,\dif t\longrightarrow\int_0^T\int_{\mt^N}\sigma(x)(\nabla \tilde{u})\psi(x,t)\,\dif x\,\dif t,\quad\tilde{\prst}\text{-a.s.}.$$
Therefore
$$\sigma(x)(\nabla \tilde{u}^n)\overset{w}{\longrightarrow}\sigma(x)(\nabla \tilde{u}),\quad\text{ in }\;L^2(\mt^N\times[0,T]),\quad\tilde{\prst}\text{-a.s.}.$$
Since any norm is weakly sequentially lower semicontinuous, it follows for all $\varphi\in C_b(\mt^N\times[0,T]\times\mr)$ and fixed $\xi\in\mr$, $\tilde{\prst}$-a.s.,
\begin{equation*}
\int_0^T\!\int_{\mt^N}\!\!\big|\sigma(x)(\nabla \tilde{u})\big|^2\varphi^2(x,t,\xi)\,\dif x\,\dif t\leq\liminf_{n\rightarrow\infty}\!\int_0^T\!\int_{\mt^N}\!\!\big|\sigma(x)(\nabla \tilde{u}^n)\big|^2\varphi^2(x,t,\xi)\dif x\dif t
\end{equation*}
and by the Fatou lemma
\begin{equation*}
\begin{split}
\int_0^T\int_{\mt^N}\int_{\mr}\big|&\sigma(x)(\nabla \tilde{u})\big|^2\varphi^2(x,t,\xi)\,\dif\delta_{\tilde{u}=\xi}\,\dif x\,\dif t\\
&\leq\liminf_{n\rightarrow\infty}\int_0^T\int_{\mt^N}\int_{\mr}\big|\sigma(x)(\nabla \tilde{u}^n)\big|^2\varphi^2(x,t,\xi)\,\dif\delta_{\tilde{u}^n=\xi}\,\dif x\,\dif t,\quad\tilde{\prst}\text{-a.s.}.
\end{split}
\end{equation*}
In other words, this gives $\tilde{n}_1=(\nabla \tilde{u})^*A(x)(\nabla\tilde{u})\delta_{\tilde{u}=\xi}\leq \tilde{o}_1\,\;\tilde{\prst}$-a.s. hence $\tilde{n}_2=\tilde{o}_2+(\tilde{o}_1-\tilde{n}_1)$ is $\tilde{\prst}$-a.s. a nonnegative measure
and the proof is complete.
\end{proof}

\end{lemma}

Let us define for all $t\in[0,T]$ and some fixed $\varphi\in C_c^\infty(\mt^N\times\mr)$
\begin{equation*}
\begin{split}
M^n(t)&=\big\langle f^n(t),\varphi\big\rangle-\big\langle f_0,\varphi\big\rangle-\int_0^t\big\langle f^n(s),b^n(\xi)\cdotp\nabla\varphi\big\rangle\,\dif s\,\\
&\quad-\int_0^t\Big\langle f^n(s),\sum_{i,j=1}^N\partial_{x_j}\big( A_{ij}(x)\partial_{x_i}\varphi\big)\Big\rangle\,\dif s-\varepsilon_n\int_0^t\big\langle f^n(s),\Delta\varphi\big\rangle\,\dif s\,\\
&\quad-\frac{1}{2}\int_0^t\big\langle \delta_{u^n=\xi}G^2_n,\partial_\xi\varphi\big\rangle\,\dif s+\big\langle m^n,\partial_\xi\varphi\big\rangle([0,t)),\qquad n\in\mn,
\end{split}
\end{equation*}
\begin{equation*}
\begin{split}
\tilde{M}^n(t)&=\big\langle \tilde{f}^n(t),\varphi\big\rangle-\big\langle f_0,\varphi\big\rangle-\int_0^t\big\langle \tilde{f}^n(s),b^n(\xi)\cdotp\nabla\varphi\big\rangle\,\dif s\,\\
&\quad-\int_0^t\Big\langle \tilde{f}^n(s),\sum_{i,j=1}^N\partial_{x_j}\big( A_{ij}(x)\partial_{x_i}\varphi\big)\Big\rangle\,\dif s-\varepsilon_n\int_0^t\big\langle \tilde{f}^n(s),\Delta\varphi\big\rangle\,\dif s\,\\
&\quad-\frac{1}{2}\int_0^t\big\langle \delta_{\tilde{u}^n=\xi}G^2_n,\partial_\xi\varphi\big\rangle\,\dif s+\big\langle\tilde{m}^n,\partial_\xi\varphi\big\rangle([0,t)),\qquad n\in\mn,
\end{split}
\end{equation*}
\begin{equation*}
\begin{split}
\tilde{M}(t)&=\big\langle \tilde{f}(t),\varphi\big\rangle-\big\langle f_0,\varphi\big\rangle-\int_0^t\big\langle \tilde{f}(s),b(\xi)\cdotp\nabla\varphi\big\rangle\,\dif s\,\\
&\quad-\int_0^t\Big\langle \tilde{f}(s),\sum_{i,j=1}^N\partial_{x_j}\big( A_{ij}(x)\partial_{x_i}\varphi\big)\Big\rangle\,\dif s-\frac{1}{2}\int_0^t\big\langle \delta_{\tilde{u}=\xi}G^2,\partial_\xi\varphi\big\rangle\,\dif s\\
&\quad+\big\langle\tilde{m},\partial_\xi\varphi\big\rangle([0,t)).
\end{split}
\end{equation*}
The proof of Theorem \ref{pass} is an immediate consequence of the following two propositions.

\begin{prop}\label{ident1}
The process $\tilde{W}$ is a $(\tilde{\mf}_t)$-cylindrical Wiener process, i.e. there exists a collection of mutually independent real-valued $(\tilde{\mf}_t)$-Wiener processes $\{\tilde{\beta}_k\}_{k\geq1}$ such that $\tilde{W}=\sum_{k\geq1}\tilde{\beta}_k e_k.$

\begin{proof}
Hereafter, times $s,t\in[0,T],\,s\leq t,$ and a continuous function
$$\gamma:C\big([0,s];H^{-1}(\mt^N)\big)\times C\big([0,s];\mathfrak{U}_0\big)\longrightarrow [0,1]$$
will be fixed but otherwise arbitrary and by $\varrho_s$ we denote the operator of restriction to the interval $[0,s]$.

Obviously, $\tilde{W}$ is a $\mathfrak{U_0}$-valued cylindrical Wiener process and is $(\tilde{\mf}_t)$-adap\-ted. According to the L\'evy martingale characterization theorem, it remains to show that it is also a $(\tilde{\mf}_t)$-martingale.
It holds true
$$\tilde{\stred}\,\gamma\big(\varrho_s \tilde{u}^n,\varrho_s\tilde{W}^n\big)\big[\tilde{W}^n(t)-\tilde{W}^n(s)\big]=\stred\,\gamma\big(\varrho_s u^n,\varrho_s W\big)\big[W(t)-W(s)\big]=0$$
since $W$ is a martingale and the laws of $(\tilde{u}^n,\tilde{W}^n)$ and $(u^n,W)$ coincide.
Next, the uniform estimate
$$\sup_{n\in\mn}\tilde{\stred}\|\tilde{W}^n(t)\|_{\mathfrak{U}_0}^2=\sup_{n\in\mn}\stred\|W(t)\|^2_{\mathfrak{U}_0}<\infty$$
and the Vitali convergence theorem yields
$$\tilde{\stred}\,\gamma\big(\varrho_s\tilde{u},\varrho_s\tilde{W}\big)\big[\tilde{W}(t)-\tilde{W}(s)\big]=0$$
which finishes the proof.
\end{proof}

\end{prop}

\begin{prop}\label{ident2}
The processes
$$\tilde{M},\qquad\tilde{M}^2-\sum_{k\geq1}\int_0^\tec\big\langle \delta_{\tilde{u}=\xi}\,g_k,\varphi\big\rangle^2\,\dif r,\qquad\tilde{M}\tilde{\beta}_k-\int_0^\tec\big\langle\delta_{\tilde{u}=\xi}\, g_k,\varphi\big\rangle\,\dif r$$
are $(\tilde{\mf}_t)$-martingales.

\begin{proof}
Here, we use the same approach and notation as the one used in the previous lemma. Let us denote by $\tilde{\beta}^n_k,\;k\geq1$ the real-valued Wiener processes corresponding to $\tilde{W}^n$, that is $\tilde{W}^n=\sum_{k\geq1}\tilde{\beta}_k^n e_k$. For all $n\in\mn$, the process
$$M^n=\int_0^\tec\big\langle\delta_{u^n=\xi}\,\varPhi^n(u^n)\dif W,\varphi\big\rangle=\sum_{k\geq1}\int_0^\tec\big\langle\delta_{u^n=\xi}\, g_k^n,\varphi\big\rangle\,\dif\beta_k(r)$$
is a square integrable $(\mf_t)$-martingale by \eqref{linrust} and by the fact that the set $\{u^n;\;n\in\mn\}$ is bounded in $L^2(\Omega;L^2(0,T;L^2(\mt^N)))$. Therefore
$$(M^n)^2-\sum_{k\geq1}\int_0^\tec\big\langle\delta_{u^n=\xi}\, g_k^n,\varphi\big\rangle^2\,\dif r,\qquad M^n\beta_k-\int_0^\tec\big\langle\delta_{u^n=\xi}\, g_k^n,\varphi\big\rangle\,\dif r$$
are $(\mf_t)$-martingales and this implies together with the equality of laws
\begin{equation}\label{exp1}
\begin{split}
\tilde{\stred}\,\gamma\big(\varrho_s \tilde{u}^n,\varrho_s\tilde{W}^n\big)\big[\tilde{M}^n(t)-\tilde{M}^n(s)\big]=\stred\,\gamma\big(\varrho_s u^n,\varrho_s W\big)\big[M^n(t)-M^n(s)\big]=0,
\end{split}
\end{equation}
\begin{equation}\label{exp2}
\begin{split}
&\tilde{\stred}\,\gamma\big(\varrho_s \tilde{u}^n,\varrho_s\tilde{W}^n\big)\bigg[(\tilde{M}^n)^2(t)-(\tilde{M}^n)^2(s)-\sum_{k\geq1}\int_s^t\big\langle\delta_{\tilde{u}^n=\xi}\, g_k^n,\varphi\big\rangle^2\,\dif r\bigg]\\
=\,&\stred\,\gamma\big(\varrho_s u^n,\varrho_s W\big)\bigg[(M^n)^2(t)-(M^n)^2(s)-\sum_{k\geq1}\int_s^t\big\langle\delta_{u^n=\xi}\, g_k^n,\varphi\big\rangle^2\,\dif r\bigg]=0,
\end{split}
\end{equation}
\begin{equation}\label{exp3}
\begin{split}
&\tilde{\stred}\,\gamma\big(\varrho_s \tilde{u}^n,\varrho_s\tilde{W}^n\big)\bigg[\tilde{M}^n(t)\tilde{\beta}_k^n(t)-\tilde{M}^n(s)\tilde{\beta}_k^n(s)-\int_s^t\big\langle \delta_{\tilde{u}^n=\xi}\,g_k^n,\varphi\big\rangle\,\dif r\bigg]\\
=\,&\stred\,\gamma\big(\varrho_s u^n,\varrho_s W\big)\bigg[M^n(t)\beta_k(t)-M^n(s)\beta_k(s)-\int_s^t\big\langle\delta_{u^n=\xi}\, g_k^n,\varphi\big\rangle\,\dif r\bigg]=0.
\end{split}
\end{equation}
Moreover, the expectations in \eqref{exp1}-\eqref{exp3} converge by the Vitali convergence theorem. Indeed, all terms are uniformly integrable by \eqref{linrust} and \eqref{energyest} and converge $\tilde{\prst}$-a.s. (after extracting a subsequence) due to Lemma \ref{weakstar}, Proposition \ref{skorokhod} and the construction of $\varPhi^\varepsilon,\,B^\varepsilon$.
Hence
\begin{equation*}
\begin{split}
&\tilde{\stred}\,\gamma\big(\varrho_s\tilde{u},\varrho_s\tilde{W}\big)\big[\tilde{M}(t)-\tilde{M}(s)\big]=0,\\
&\tilde{\stred}\,\gamma\big(\varrho_s\tilde{u},\varrho_s\tilde{W}\big)\bigg[\tilde{M}^2(t)-\tilde{M}^2(s)-\sum_{k\geq1}\int_s^t\big\langle\delta_{\tilde{u}=\xi}\, g_k,\varphi\big\rangle^2\dif r\bigg]=0,\\
&\tilde{\stred}\,\gamma\big(\varrho_s\tilde{u},\varrho_s\tilde{W}\big)\bigg[\tilde{M}(t)\tilde{\beta}_k(t)-\tilde{M}(s)\tilde{\beta}_k(s)-\int_s^t\big\langle\delta_{\tilde{u}=\xi}\, g_k,\varphi\big\rangle\dif r\bigg]=0,
\end{split}
\end{equation*}
which gives the $(\tilde{\mf}_t)$-martingale property.
\end{proof}

\end{prop}

\begin{proof}[Proof of Theorem \ref{pass}]
Once the Propositions \ref{ident1}, \ref{ident2} are established, it follows that the quadratic variation
$$\bigg\langle\!\!\!\bigg\langle\tilde{M}-\int_0^\tec\big\langle\delta_{\tilde{u}=\xi}\,\varPhi(\tilde{u})\,\dif\tilde{W},\varphi\big\rangle\bigg\rangle\!\!\!\bigg\rangle=0$$
and so for every $\varphi\in C_c^\infty(\mt^N\times\mr),\;t\in[0,T],\;\tilde{\prst}\text{-a.s.}$
\begin{equation*}\label{reseni}
\begin{split}
\big\langle \tilde{f}(t),&\varphi\big\rangle-\big\langle f_0,\varphi\big\rangle-\int_0^t\!\big\langle \tilde{f}(s),b(\xi)\cdotp\nabla\varphi\big\rangle\dif s-\int_0^t\!\Big\langle \tilde{f}(s),\sum_{i,j=1}^N\!\partial_{x_j}\big( A_{ij}(x)\partial_{x_i}\varphi\big)\Big\rangle\dif s\\
&=\int_0^t\big\langle\delta_{\tilde{u}=\xi}\,\varPhi(\tilde{u})\,\dif\tilde{W},\varphi\big\rangle+\frac{1}{2}\int_0^t\big\langle \delta_{\tilde{u}=\xi}G^2,\partial_\xi\varphi\big\rangle\,\dif s-\big\langle\tilde{m},\partial_\xi\varphi\big\rangle([0,t))
\end{split}
\end{equation*}
and the statement follows.
\end{proof}

\subsection{Pathwise solutions}

In order to finish the proof, we make use of the Gy\"{o}ngy-Krylov characterization of convergence in probability introduced in \cite{krylov}. It is useful in situations when the pathwise uniqueness and the existence of at least one martingale solution imply the existence of a unique pathwise solution.

\begin{prop}\label{diagonal}
Let $X$ be a Polish space equipped with the Borel $\sigma$-algebra. A sequence of $X$-valued random variables $\{Y_n;\,n\in\mn\}$ converges in probability if and only if for every subsequence of joint laws, $\{\mu_{n_k,m_k};\,k\in\mn\}$, there exists a further subsequence which converges weakly to a probability measure $\mu$ such that
$$\mu\big((x,y)\in X\times X;\,x=y\big)=1.$$
\end{prop}

We consider the collection of joint laws of $(u^n,u^m)$, denoted by $\mu_u^{n,m}$. For this purpose we define the extended phase space (cf. \eqref{pathspace})
$$\mathcal{X}^J=\mathcal{X}_u\times\mathcal{X}_u\times\mathcal{X}_W,\qquad\mathcal{X}^J_u=\mathcal{X}_u\times\mathcal{X}_u,$$
As above, denote by $\mu_u^n$ the law of $u^n$ and by $\mu_W$ the law of $W$. Set further
$$\mu_u^{n,m}=\mu_u^n\otimes\mu_u^m,\qquad\nu^{n,m}=\mu_u^n\otimes\mu_u^m\otimes\mu_W.$$

Similarly to Proposition \ref{tight} the following fact holds true. The proof is nearly identical and so will be left to the reader.

\begin{prop}
The collection $\{\nu^{n,m};\,n,m\in\mn\}$ is tight on $\mathcal{X}^J$.
\end{prop}

Let us take any subsequence $\{\nu^{n_k,m_k};\,k\in\mn\}$. By the Prokhorov theorem, it is relatively weakly compact hence it contains a weakly convergent subsequence. Without loss of generality we may assume that the original sequence $\{\nu^{n_k,m_k};\,k\in\mn\}$ itself converges weakly to a measure $\nu$. According to the Skorokhod representation theorem, we infer the existence of a probability space $(\bar{\Omega},\bar{\mf},\bar{\prst})$ with a sequence of random variables $(\hat{u}^{n_k},\check{u}^{m_k},\bar{W}^k),k\in\mn,$ conver\-ging almost surely in $\mathcal{X}^J$ to a random variable $(\hat{u},\check{u},\bar{W})$ and
$$\bar{\prst}\big((\hat{u}^{n_k},\check{u}^{m_k},\bar{W}^k)\in\,\,\cdotp\big)=\nu^{n_k,m_k}(\cdot),\qquad\bar{\prst}\big((\hat{u},\check{u},\bar{W})\in\,\,\cdotp\big)=\nu(\cdot).$$
Observe that in particular, $\mu_u^{n_k,m_k}$ converges weakly to a measure $\mu_u$ defined by
$$\mu_u(\cdot)=\bar{\prst}\big((\hat{u},\check{u})\in\,\,\cdotp\big).$$
As the next step, we should recall the technique established in the previous section. Analogously, it can be applied to both $(\hat{u}^{n_k},\bar{W}^k),\,(\hat{u},\bar{W})$ and $(\check{u}^{m_k},\bar{W}^k),\,(\check{u},\bar{W})$ in order to show that $(\hat{u},\bar{W})$ and $(\check{u},\bar{W})$ are martingale kinetic solutions of \eqref{rovnice} defined on the same stochastic basis $(\bar{\Omega},\bar{\mf},(\bar{\mf}_t),\bar{\prst})$, where
$$\bar{\mf}_t=\sigma\big(\sigma\big(\varrho_t\hat{u},\varrho_t\check{u},\varrho_t\bar{W}\big)\cup\big\{N\in\bar{\mf};\;\bar{\prst}(N)=0\big\}\big),\quad t\in[0,T].$$
Since $\hat{u}(0)=\check{u}(0)\;\,\bar{\prst}$-a.s., we conclude from Theorem \ref{uniqueness} that $\hat{u}=\check{u}$ in $\mathcal{X}_u$ $\bar{\prst}$-a.s. hence
$$\mu_u\big((x,y)\in\mathcal{X}_u\times\mathcal{X}_u;\;x=y\big)=\bar{\prst}\big(\hat{u}=\check{u}\text{ in }\mathcal{X}_u\big)=1.$$
Now, we have all in hand to apply Proposition \ref{diagonal}. It implies that the original sequence $u^n$ defined on the initial probability space $(\Omega,\mf,\prst)$ converges in probability in the topology of $\mathcal{X}_u$ to a random variable $u$. Without loss of gene\-rality, we assume that $u^n$ converges to $u$ almost surely in $\mathcal{X}_u$ and again by the method from section \ref{passage} we finally deduce that $u$ is a pathwise kinetic solution to \eqref{rovnice}. Actually, identification of the limit is more straightforward here since in this case all the work is done for the initial setting and only one fixed driving Wiener process $W$ is considered.

\section{Existence - general initial data}

In this final section we provide an existence proof in the general case of $u_0\in L^p(\Omega;L^p(\mt^N)),$ for all $p\in[1,\infty)$. It is a straightforward consequence of the previous section.
We approximate the initial condition by a sequence $\{u^\varepsilon_0\}\subset L^p(\Omega;C^\infty(\mt^N)),$ $p\in[1,\infty),$ such that $u_0^\varepsilon\rightarrow u_0$ in $L^1(\Omega;L^1(\mt^N))$. That is, the initial condition $u_0^\varepsilon$ can be defined as a pathwise mollification of $u_0$ so that it holds true
\begin{equation}\label{iio}
\|u^\varepsilon_0\|_{L^p(\Omega;L^p(\mt^N))}\leq\|u_0\|_{L^p(\Omega;L^p(\mt^N))},\qquad\varepsilon\in(0,1),\;p\in[1,\infty).
\end{equation}
According to the previous section, for each $\varepsilon\in(0,1)$, there exists a kinetic solution $u^\varepsilon$ to \eqref{rovnice} with initial condition $u_0^\varepsilon$.
By application of the comparison principle \eqref{comparison},
$$\stred\|u^{\varepsilon_1}-u^{\varepsilon_2}\|_{L^1(0,T;L^1(\mt^N))}\leq T\,\stred\|u^{\varepsilon_1}_0-u^{\varepsilon_2}_0\|_{L^1(\mt^N)},\quad \varepsilon_1,\varepsilon_2\in(0,1).$$
Therefore, $\{u^\varepsilon;\,\varepsilon\in(0,1)\}$ is a Cauchy sequence in $L^1(\Omega;L^1(0,T;L^1(\mt^N)))$ and there exists $u\in L^1(\Omega;L^1(0,T;L^1(\mt^N)))$ such that
$$u^\varepsilon\longrightarrow u\quad\text{in}\quad L^1(\Omega;L^1(0,T;L^1(\mt^N))).$$
By \eqref{iio}, we still have the uniform energy estimates
\begin{equation}\label{en}
\begin{split}
\stred\sup_{0\leq t\leq T}\|u^\varepsilon(t)\|^p_{L^p(\mt^N)}\leq C_{T,u_0},\quad p\in[2,\infty),
\end{split}
\end{equation}
as well as (using the usual notation)
\begin{equation}\label{dd}
\stred\big|m^\varepsilon(\mt^N\times[0,T]\times\mr)\big|^2\leq C_{T,u_0}.
\end{equation}
Thus, using this observations as in Lemma \ref{weakstar}, one finds that there exists a subsequence $\{u^n;\,n\in\mn\}$ such that
\begin{enumerate}
 \item $f^n\overset{w^*}{\longrightarrow}f\;\quad\text{in}\quad L^\infty(\Omega\times\mt^N\times[0,T]\times\mr)\text{-weak}^*,$
 \item there exists a kinetic measure $m$ such that
$$m^n\overset{w^*}{\longrightarrow}m\;\quad\text{in}\quad L^2(\Omega;\mathcal{M}_b)\text{-weak}^*$$
and $m=n_1+n_2$, where $\,n_1=(\nabla u)^*A(x)(\nabla u)\delta_{u=\xi}$ and $n_2$ is almost surely a nonnegative measure over $\mt^N\times[0,T]\times\mr$.
\end{enumerate}
With these facts in hand, we are ready to pass to the limit in \eqref{kinet} and conclude that $u$ satisfies the kinetic formulation in the sense of distributions.
Note, that \eqref{en} remains valid also for $u$ so \eqref{integrov} follows and the proof of Theorem \ref{main} is complete.

\section*{Acknowledgments}
The author wishes to thank Arnaud Debussche and Jan Seidler for many discussions and valuable suggestions.

\end{document}